\documentclass{agtart_a}
\pdfoutput=1

\usepackage{pinlabel,amscd}

\title[Cobordism of Morse functions on surfaces and applications]
{Cobordism of Morse functions on surfaces,\\
the universal complex of singular fibers\\
and their application to map germs}

\author{Osamu Saeki}
\givenname{Osamu}
\surname{Saeki}
\address{Faculty of Mathematics\\
Kyushu University\\\newline
Hakozaki\\Fukuoka 812-8581\\Japan}
\email{saeki@math.kyushu-u.ac.jp}
\urladdr{http://www.math.kyushu-u.ac.jp/~saeki/}

\volumenumber{6}
\issuenumber{}
\publicationyear{2006}
\papernumber{19}
\startpage{539}
\endpage{572}

\doi{}
\MR{}
\Zbl{}


\received{22 September 2005}
\revised{}
\accepted{25 January 2006}
\published{7 April 2006}
\publishedonline{7 April 2006}
\proposed{}
\seconded{}
\corresponding{}
\editor{}
\version{}

\arxivreference{} 

\AtBeginDocument{\let\bar\wbar\let\tilde\wtilde\let\hat\what}
\def\S{Section }

\makeatletter
\def\cnewtheorem#1[#2]#3{\newtheorem{#1}{#3}[section]
\expandafter\let\csname c@#1\endcsname\c@thm}

\newtheorem{thm}{Theorem}[section]
\cnewtheorem{cor}[thm]{Corollary}
\cnewtheorem{lem}[thm]{Lemma}
\cnewtheorem{prop}[thm]{Proposition}
\theoremstyle{definition}
\cnewtheorem{dfn}[thm]{Definition}
\cnewtheorem{notn}[thm]{Notation}
\theoremstyle{remark}
\cnewtheorem{exam}[thm]{Example}
\cnewtheorem{rmk}[thm]{Remark}
\cnewtheorem{prob}[thm]{Problem}
\numberwithin{equation}{section}

\makeatother  

\newcommand{\Int}{\mathop{\mathrm{Int}}\nolimits}

\newcommand{\id}{\mathop{\mathrm{id}}\nolimits}

\def\spmapright#1{\smash{%
 \mathop{\hbox to 1.3cm{\rightarrowfill}}
  \limits^{#1}}}

\begin{document}

\begin{asciiabstract}
We give a new and simple proof for the computation of the oriented and
the unoriented fold cobordism groups of Morse functions on
surfaces. We also compute similar cobordism groups of Morse functions
based on simple stable maps of 3-manifolds into the plane.
Furthermore, we show that certain cohomology classes associated with
the universal complexes of singular fibers give complete invariants
for all these cobordism groups.  We also discuss invariants derived
from hypercohomologies of the universal homology complexes of singular
fibers. Finally, as an application of the theory of universal
complexes of singular fibers, we show that for generic smooth map
germs g: (R^3, 0) --> (R^2, 0) with R^2 being oriented, the algebraic
number of cusps appearing in a stable perturbation of g is a local
topological invariant of g.
\end{asciiabstract}

\begin{htmlabstract}
We give a new and simple proof for the computation of the oriented and
the unoriented fold cobordism groups of Morse functions on
surfaces. We also compute similar cobordism groups of Morse functions
based on simple stable maps of 3&ndash;manifolds into the
plane. Furthermore, we show that certain cohomology classes associated
with the universal complexes of singular fibers give complete
invariants for all these cobordism groups.  We also discuss invariants
derived from hypercohomologies of the universal homology complexes of
singular fibers. Finally, as an application of the theory of universal
complexes of singular fibers, we show that for generic smooth map
germs <i>g</i>: (<b>R</b>&sup3;, 0) &#x2192; (<b>R</b>&sup2;, 0) with
<b>R</b>&sup2; being oriented, the algebraic number of cusps appearing
in a stable perturbation of <i>g</i> is a local topological invariant of
<i>g</i>.
\end{htmlabstract}

\begin{abstract}
We give a new and simple proof for the computation of the oriented and
the unoriented fold cobordism groups of Morse functions on
surfaces. We also compute similar cobordism groups of Morse functions
based on simple stable maps of $3$--manifolds into the
plane. Furthermore, we show that certain cohomology classes associated
with the universal complexes of singular fibers give complete
invariants for all these cobordism groups.  We also discuss invariants
derived from hypercohomologies of the universal homology complexes of
singular fibers. Finally, as an application of the theory of universal
complexes of singular fibers, we show that for generic smooth map
germs $g \colon\, (\mathbb{R}^3, 0) \to (\mathbb{R}^2, 0)$ with
$\mathbb{R}^2$ being oriented, the algebraic number of cusps appearing
in a stable perturbation of $g$ is a local topological invariant of
$g$.
\end{abstract}

\maketitle

\section{Introduction}\label{section1}

In \cite{RS}, Rim\'anyi and Sz\H{u}cs
introduced the notion of a
cobordism for singular maps.
In fact, in the classical work
of Thom \cite{Thom}, one can find
the notion of a cobordism for
embeddings, and they naturally generalized
this concept to differentiable maps with
prescribed local and global singularities.
In particular, when the dimension of
the target manifold is greater than
or equal to that of the source manifold,
they described the cobordism groups in
terms of the homotopy groups of
a certain universal space by means of a
Pontrjagin--Thom type construction.

However, when the dimension of
the target is strictly smaller
than that of the source, their
method cannot be directly applied.
In \cite{Saeki932}, the author
defined the cobordism group
for maps with only definite fold
singularities, and used a geometric
argument in \cite{Saeki021} to show
that the cobordism group of
such \emph{functions} is isomorphic to
the $h$--cobordism
group of homotopy spheres (Kervaire--Milnor \cite{KM}). 
Recently this result was generalized for
\emph{maps} by Sadykov \cite{Sadykov}
with the aid of a Pontrjagin--Thom type construction.
This was possible, since the class of
singularities is quite restricted and the structure
of regular fibers of such maps is well-understood.

On the other hand, Ikegami and the
author \cite{IS} defined and studied the oriented
(fold) cobordism group of Morse
functions on surfaces. Since a (fold)
cobordism between Morse functions on
closed surfaces does not allow cusp singularities,
this cobordism group is rather bigger
than the usual cobordism
group of $2$--dimensional manifolds.
In fact, Ikegami and the author employed
a geometric argument using functions
on finite graphs to show that the
group is in fact an infinite cyclic group.
Recently, the structure of the unoriented cobordism
group of Morse functions on surfaces
was determined by Kalm\'ar \cite{Kalmar}
by using a similar method. In \cite{Ikegami}
Ikegami determined the structures of the oriented
and the unoriented cobordism groups
of Morse functions on manifolds of
arbitrary dimensions by using
an argument employing the cusp elimination technique
based on Levine \cite{Levine0}.

On the other hand, in an attempt to
construct a rich family of
cobordism invariants for 
maps with prescribed local
and global singularities in the
case where the dimension of the
target is strictly smaller than
that of the source, the author
considered singular fibers of
such maps and developed the theory
of universal complexes of singular
fibers \cite{Saeki04}.
The terminology ``singular fiber'' here
refers to a certain right-left
equivalence class of a map germ
along the inverse image of a point
in the target. The equivalence
classes of such singular fibers
together with their adjacency relations
lead to a cochain complex, and
in \cite{Saeki04} it was shown that
its cohomology classes give rise
to cobordism invariants for
singular maps. In fact, the
isomorphism between the
oriented cobordism group of Morse
functions on surfaces and the infinite
cyclic group was reconstructed by
using an invariant derived from
the universal complex of singular fibers
\cite[\S14.2]{Saeki04}.

This paper has three purposes.
The first one (see Sections \ref{section2} and \ref{section3})
is to give a new
and simple proof for the
calculation of the oriented
and the unoriented fold cobordism
groups of Morse functions on surfaces
(\fullref{thm:main1}).
In Ikegami--Saeki \cite{IS} and Kalm\'ar \cite{Kalmar}, the calculation
was done by simplifying a given
function on a graph by employing certain
moves. In this paper,
we also use the same moves, but
the simplification is drastically simple.

Furthermore, we will introduce the notion
of the oriented
and the unoriented \emph{simple} fold cobordism groups
of Morse functions on surfaces
by restricting the fold cobordisms
to simple ones (for simple maps,
the reader is referred to \cite{Saeki2.5,Saeki4} or \fullref{dfn:simple}
of the present paper). We will show that the
simple cobordism groups are in fact isomorphic
to the corresponding cobordism groups
(\fullref{thm:main2}).

The second purpose of this paper
is to construct cobordism invariants
for the four cobordism groups of
Morse functions on surfaces considered
above by using cohomology classes of the universal
complexes of singular fibers as developed
in \cite{Saeki04} (see
Sections \ref{section4} and \ref{section5}). 
By combining
the universal complex of co-orientable
singular fibers (with coefficients in $\Z$)
and that of usual (not necessarily
co-orientable) singular fibers
(with coefficients in $\Z_2$), we will
obtain complete cobordism invariants
for all the four cases above.

In \cite{Kazarian}, Kazarian constructed
the universal homology complex of singularities
by combining the universal complex
of co-orientable singularities and
that of usual (not necessarily
co-orientable) singularities
defined by Vassiliev \cite{V},
and studied their hypercohomologies.
In \fullref{section5}, we will consider the
analogy of Kazarian's construction
in our situation of singular fibers.
We will see that the hypercohomology classes
give rise to cobordism invariants,
but for cobordism groups of
Morse functions on surfaces, we
obtain the same invariants as
those obtained by using the usual 
universal complex of singular fibers.
It would be interesting to
study the hypercohomologies of higher
dimensional analogues to see if
there is a ``hidden singular
fiber'' in a sense similar to
Kazarian's \cite{Kazarian}.

The third purpose of this paper
is to give an application of the
theory of universal complexes of
singular fibers developed in Sections \ref{section4}
and \ref{section5} to the theory
of stable perturbations of map germs.
More precisely, we will consider
a smooth map germ $g \co (\R^3, 0)
\to (\R^2, 0)$ which is generic
in the sense of Fukuda and Nishimura \cite{Fukuda1,Nishimura}.
Then a stable
perturbation $\tilde{g}$ of a representative
of $g$ has
isolated cusps, and for each
cusp singular point we can define a sign $+1$
or $-1$ by using the indices of the
nearby fold points together with
the orientation
of the target $\R^2$.
Then by using the theory of singular fibers
(of generic maps of $3$--dimensional manifolds with
boundary into the $2$--dimensional disk),
we will show that the algebraic number of
cusps of $\tilde{g}$ is a local topological
invariant of $g$ with $\R^2$ being
oriented (\fullref{thm:new}). 
We also describe
this integer by a certain
cobordism invariant of a $C^\infty$
stable map of a compact surface
with boundary into $S^1$ associated
with $g$. This invariant is strongly
related to the isomorphism between
the oriented cobordism group
of Morse functions on surfaces and
the infinite cyclic group
constructed in Sections \ref{section3}
and \ref{section4}. It would be
interesting to compare our result
with those obtained by Fukuda and Ishikawa \cite{FI},
Fukui, Nu{\~n}o~Ballesteros and Saia \cite{FBS},
Nu{\~n}o~Ballesteros and Saia \cite{BS}, and Ohsumi \cite{Ohsumi},
about the number of certain singularities
appearing in a stable perturbation 
(or a generic deformation) of
a given map germ.

Throughout the paper, manifolds and maps
are differentiable of class $C^\infty$
unless otherwise indicated.
For a topological space $X$, $\id_X$ denotes
the identity map of $X$. 

The author would like to thank Boldizs\'ar Kalm\'ar
and Toshizumi Fukui
for stimulating discussions.
The author has been supported in part by
Grant-in-Aid for Scientific Research
(No.~16340018), Japan Society for the Promotion of Science.

\section{Preliminaries}\label{section2}

In this section, we recall some basic notions
about smooth functions and maps, and state two
theorems about cobordism groups of
Morse functions on surfaces.

A smooth real-valued function on a smooth manifold 
is called a \emph{Morse function} if its critical
points are all non-degenerate.
We do not assume that the values at the
critical points are all distinct: distinct critical
points may have the same value. If the
critical values are all distinct, then such a Morse
function is said to be ($C^\infty$)
\emph{stable}.
For details, see Golubitsky and Guillemin \cite[Chapter III, \S2]{GG}.

For a positive integer $n$, we denote by $M^{SO}(n)$ 
(or $M(n)$) the
set of all Morse functions on closed
oriented (resp.\ possibly nonorientable)
$n$--dimensional manifolds.
We adopt the convention that the function on the empty
set $\emptyset$ is an element of $M^{SO}(n)$ and of
$M(n)$ for all $n$. 

Before defining the cobordism groups of Morse functions,
let us recall the notion of fold singularities.
Let $f \co M \to N$ be a smooth map between smooth
manifolds with $n = \dim M \geq \dim N = p$.
A \emph{singular point} of $f$ is a point $q \in M$
such that the rank of the differential $df_q \co T_qM \to T_{f(q)}N$
is strictly smaller than $p$. We denote
by $S(f)$ the set of all singular points of $f$ and
call it the \emph{singular set} of $f$.
A singular point $q \in S(f)$ is a \emph{fold singular point}
(or a \emph{fold point})
if there exist local coordinates $(x_1, x_2, \ldots, x_n)$
and $(y_1, y_2, \ldots, y_p)$ around $q$ and $f(q)$
respectively such that $f$ has the form
$$y_i \circ f = \left\{\begin{array}{l}
x_i, \quad 1 \leq i \leq p-1, \\
\pm x_p^2 \pm x_{p+1}^2 \pm \cdots \pm x_n^2, \quad i=p.
\end{array}\right.$$
If the signs appearing in $y_p \circ f$ all
coincide, then we say that $q$ is a \emph{definite
fold singular point} (or a \emph{definite fold point}), 
otherwise an \emph{indefinite fold singular point}
(or an \emph{indefinite fold point}).

If a smooth map $f$ has only fold points as its
singularities, then we say that $f$ is a \emph{fold
map}.

\begin{dfn}\label{dfn:cob}
Two Morse functions $f_0 \co M_0 \to \R$ and $f_1 \co M_1 \to \R$
in $M^{SO}(n)$ are said to be \emph{oriented cobordant} (or
\emph{oriented fold cobordant}) if
there exist a compact oriented $(n+1)$--dimensional
manifold $X$ and a fold map 
$F\co X \to \R \times [0, 1]$ such that
\begin{itemize}
\item[(1)] the oriented boundary $\partial X$ of $X$
is the disjoint union $M_0 \amalg (-M_1)$, where $-M_1$ denotes
the manifold $M_1$ with the orientation reversed, and
\item[(2)] we have
\begin{align*}
F|_{M_0 \times [0, \varepsilon)} & = f_0 \times \id_{[0,
\varepsilon)} \co M_0 \times [0, \varepsilon) \to \R \times
[0, \varepsilon), \quad \text{\rm and} \\
F|_{M_1 \times (1-\varepsilon, 1]} 
& = f_1 \times \id_{(1-\varepsilon,
1]} \co M_1 \times (1-\varepsilon, 1] \to \R \times
(1-\varepsilon, 1]
\end{align*}
for some sufficiently small $\varepsilon > 0$, where
we identify the open 
collar neighborhoods of $M_0$ and $M_1$ in
$X$ with $M_0 \times [0, \varepsilon)$ and $M_1 
\times (1-\varepsilon,
1]$ respectively.
\end{itemize}
In this case, we call $F$ an \emph{oriented cobordism} between
$f_0$ and $f_1$.

If a Morse function in $M^{SO}(n)$ is oriented cobordant to
the function on the empty set, then we say that it
is \emph{oriented null-cobordant}.

It is easy to show that the above relation 
defines an equivalence
relation on the set $M^{SO}(n)$ for each $n$. 
Furthermore, we see easily that the set of all
oriented cobordism classes forms an additive group
under disjoint union: the neutral element is the class
corresponding to oriented null-cobordant Morse functions, and the
inverse of a class represented by a Morse function 
$f \co M \to \R$ is given by
the class of $-f \co -M \to \R$.
We denote
by $\mathcal{M}^{SO}(n)$ the group of all oriented (fold) cobordism classes
of elements of $M^{SO}(n)$ and call it the \emph{oriented}
(\emph{fold}) \emph{cobordism
group of Morse functions on manifolds of dimension $n$},
or the \emph{$n$--dimensional oriented cobordism group of Morse functions}.

We can also define the unoriented versions of all the
objects defined above by forgetting the orientations and
by using $M(n)$ instead of $M^{SO}(n)$.
For the terminologies,
we omit the term ``oriented'' (or use ``unoriented''
instead) for the corresponding unoriented versions. 
The unoriented cobordism group of Morse functions
on manifolds of dimension $n$ is denoted by 
$\mathcal{M}(n)$ by omitting the superscript $SO$.
\end{dfn}

\begin{rmk}
The oriented cobordism group $\mathcal{M}^{SO}(n)$
is denoted by $\mathcal{M}(n)$ in Ikegami--Saeki \cite{IS}
and by $\mathcal{M}_n$ in Ikegami \cite{Ikegami}.
Furthermore, the unoriented cobordism group
$\mathcal{M}(n)$ is denoted by $\mathcal{N}_n$ in
\cite{Ikegami} and by
$Cob_f(n, 1-n)$ in Kalm\'ar \cite{Kalmar}.
\end{rmk}

\begin{rmk}\label{rmk:stable}
Let $M$ be a closed (oriented) $n$--dimensional manifold.
It is easy to see that
if two Morse functions $f$ and $g$ on $M$ are connected
by a one-parameter family of Morse functions, then they
are (oriented) cobordant. In particular, every Morse function
is (oriented) cobordant to a stable Morse function. 
\end{rmk}

The following isomorphisms have been proved in
\cite{Ikegami,IS,Kalmar}.

\begin{thm}\label{thm:main1}
$(1)$ The $2$--dimensional oriented cobordism group of Morse
functions $\mathcal{M}^{SO}(2)$ is isomorphic to $\Z$,
the infinite cyclic group.

$(2)$ The $2$--dimensional unoriented cobordism group of Morse
functions $\mathcal{M}(2)$ is isomorphic to
$\Z \oplus \Z_2$.
\end{thm}

In \fullref{section3}, we will give a new and simple proof
for the above isomorphisms.
We will also describe explicit isomorphisms.

In order to define new cobordism groups of
Morse functions, we need the following.

\begin{dfn}\label{dfn:simple}
Let $f \co M \to N$ be a smooth map between smooth manifolds.
We say that $f$ is \emph{simple} if for every $y \in N$,
each connected component
of $f^{-1}(y)$ contains at most one singular point.
\end{dfn}

Note that a stable Morse function is always
simple.

If $\dim{N} = 2$ and $f \co M \to N$
is a fold map which is $C^\infty$ stable (for details, see
Golubitsky--Guillemin, Levine, Saeki 
\cite{GG,Levine1,Saeki04}), then
$S(f)$ is a regular $1$--dimensional
submanifold of $M$ and the map
$f|_{S(f)}$ is an immersion with normal
crossings.
Therefore, for every $y \in N$,
$f^{-1}(y)$ contains at most two singular points.
Fold maps of closed $3$--manifolds into the plane
which are $C^\infty$ stable and simple
have been studied by the author 
\cite{Saeki2.5,Saeki4}. 

\begin{dfn}
For a positive integer $n$, we denote by $SM^{SO}(n)$ 
(or $SM(n)$) the
set of all \emph{stable} 
Morse functions on closed
oriented (resp.\ possibly nonorientable)
$n$--dimensional manifolds.

Two stable Morse functions 
$f_0 \co M_0 \to \R$ and $f_1 \co M_1 \to \R$
in $SM^{SO}(n)$
are said to be \emph{simple oriented} (\emph{fold})
\emph{cobordant} if
there exist an oriented cobordism $F$ between
$f_0$ and $f_1$ as in \fullref{dfn:cob}
such that $F$ is simple and $F|_{S(F)}$
is an immersion with normal crossings.
In this case, we call $F$ a \emph{simple
oriented cobordism} between
$f_0$ and $f_1$.

It is easy to show that the above relation 
defines an equivalence
relation on the set $SM^{SO}(n)$ for each $n$. 
Furthermore, we see easily that the set of all
simple oriented cobordism classes forms an additive group
under disjoint union.
We denote
by $\mathcal{SM}^{SO}(n)$ the group of all simple
oriented (fold) cobordism classes
of elements of $SM^{SO}(n)$ and call it the \emph{simple oriented}
(\emph{fold}) \emph{cobordism
group of Morse functions on manifolds of dimension $n$},
or the \emph{$n$--dimensional simple oriented 
cobordism group of Morse functions}.

We can also define the unoriented versions of all the
objects defined above.
The simple unoriented cobordism group of Morse functions
on manifolds of dimension $n$ is denoted by 
$\mathcal{SM}(n)$.
\end{dfn}

We will prove the following in \fullref{section3}.

\begin{thm}\label{thm:main2}
$(1)$ The $2$--dimensional simple oriented cobordism group of Morse
functions $\mathcal{SM}^{SO}(2)$ is isomorphic to $\Z$,
the infinite cyclic group.

$(2)$ The $2$--dimensional simple unoriented cobordism group of Morse
functions $\mathcal{SM}(2)$ is isomorphic to
$\Z \oplus \Z_2$.
\end{thm}

In fact, we will see that
the natural homomorphisms
$$\mathcal{SM}^{SO}(2) \to \mathcal{M}^{SO}(2)
\quad \text{ and } \quad
\mathcal{SM}(2) \to \mathcal{M}(2)$$
are isomorphisms.

\section{Proofs of the theorems}\label{section3}

In this section, we will prove Theorems~\ref{thm:main1}
and \ref{thm:main2}.

Let us first recall
the following notion
of a Stein factorization
(for more details, see \cite{KLP,Levine1,Saeki04},
for example).

\begin{dfn}
Suppose that a smooth map $f \co M \to N$ with $n = \dim M 
\geq \dim N = p$
is given. Two points in $M$ are \emph{equivalent} with respect to $f$
if they lie on the same component of an $f$--fiber. Let $W_f$ denote
the quotient space of $M$ with respect to this equivalence relation and
$q_f \co M \to W_f$ the quotient map. It is easy to see that then
there exists a unique continuous map $\bar{f} \co W_f \to N$ such that
$f = \bar{f} \circ q_f$. The space $W_f$ or the commutative diagram
\begin{eqnarray*}
& M \spmapright{f} N & \\
& \scriptstyle{q_f} \searrow \,\quad \nearrow \scriptstyle{\bar{f}} & \\
& W_f &
\end{eqnarray*}
is called the \emph{Stein factorization} of $f$.
\end{dfn}

If $f \co M \to \R$ is a Morse function on a closed
manifold $M$, then $W_f$ has the natural structure
of a $1$--dimensional CW complex. In this case, we call
$W_f$ the \emph{Reeb graph} of $f$ (for example, see \cite{Fomenko}).
Furthermore, we call the continuous map $\bar{f} \co W_f \to \R$
a \emph{Reeb function}.

Let $f \co M \to \R $ be a stable Morse function on a closed
(possibly nonorientable) surface $M$. 
Then by \cite[Lemma~3.2]{IS} and \cite[\S3]{Kalmar},
its Reeb graph $W_f$ is a finite graph whose
vertices are the $q_f$--images of
the critical points of $f$ such that
\begin{itemize}
\item[$(1)$] the vertices corresponding to critical points
of index $0$ or $2$ have degree $1$, and those of index $1$
have degree $2$ or $3$,
\end{itemize}
and the Reeb function $\bar{f} \co W_f \to \R$ satisfies
the following:
\begin{itemize}
\item[$(2)$] around each vertex of $W_f$, $\bar{f}$ is 
equivalent to one of the functions as
depicted in \fullref{fig1}, and
\item[$(3)$] $\bar{f}$ is an embedding on each edge.
\end{itemize}
Furthermore, a degree $2$ vertex occurs only if $M$ is
nonorientable.

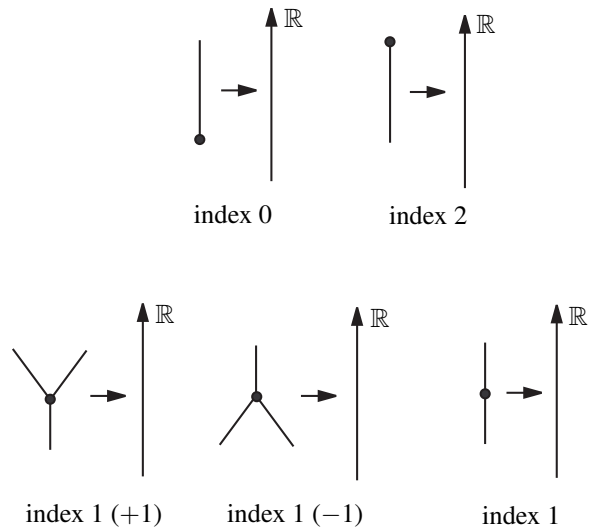
\begin{figure}[ht!]\small
\begin{center}
\labellist
\pinlabel {$\R$} [lB] at 255 743
\pinlabel {$\R$} [lB] at 536 734
\pinlabel {$\R$} [lB] at 67 311
\pinlabel {$\R$} [lB] at 381 305
\pinlabel {$\R$} [lB] at 671 308
\pinlabel {index $0$} at 192 469
\pinlabel {index $2$} at 477 469
\pinlabel {index $1$ $(+1)$} at -11 30
\pinlabel {index $1$ $(-1)$} at 283 30
\pinlabel {index $1$} at 613 30
\endlabellist
\includegraphics[width=0.6\linewidth]{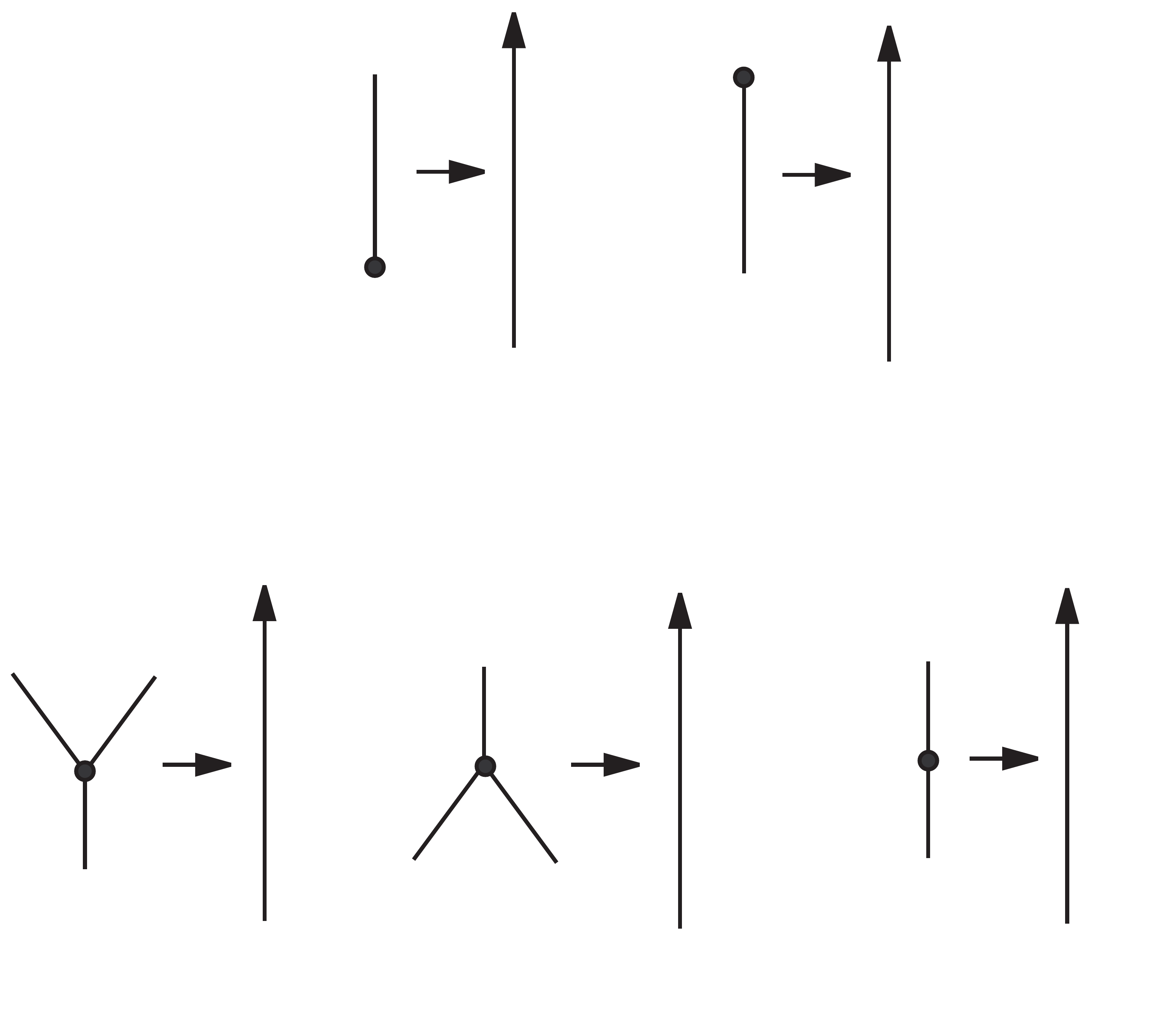}
\end{center}
\caption{Behavior of $\bar{f}$ around each vertex of the Reeb graph $W_f$}
\label{fig1}
\end{figure}

To each vertex of degree three of a Reeb graph
we associate the sign $+1$ or $-1$ as in \fullref{fig1}.

For the proof of \fullref{thm:main1},
in \cite{IS,Kalmar}, certain moves for
Reeb functions have been considered
(for details, see \cite[Figure~3]{IS} and
\cite[Figure~3]{Kalmar}). Recall that
these moves correspond to the Stein
factorizations of cobordisms between
Morse functions.
We will refer to the types of these
moves according to \cite[Figure~3]{Kalmar}
(types (a)--(k)). 
We call the following
types \emph{admissible moves}
for each of the four situations:
\begin{itemize}
\item[$(1)$] oriented cobordism: (a)--(g),
\item[$(2)$] unoriented cobordism: (a)--(k),
\item[$(3)$] simple oriented cobordism: (a)--(d),
\item[$(4)$] simple unoriented cobordism:
(a)--(d), (h), (i).
\end{itemize}
According to \cite{IS,Kalmar},
we have only to show that the Reeb
function $\bar{f} \co W_f \to \R$ associated
with an arbitrary stable Morse function
$f \co M \to \R$ on a closed (orientable) surface 
can be deformed to a standard form
(see \cite[Figure~5]{IS} and
\cite[Figures~4 and 5]{Kalmar})
by a finite iteration of admissible moves.

Let us first apply the move (a) to each
edge of $W_f$ so that $\bar{f}$ decomposes
into a disjoint union of four elementary
functions as depicted in \fullref{fig3}.
Note that if $M$ is orientable, then
the piece as in \fullref{fig3} (4) 
does not appear.

\begin{figure}[ht!]\small
\begin{center}
\labellist
\pinlabel {$\R$} [Bl] at 79 697
\pinlabel {$\R$} [Bl] at 224 697
\pinlabel {$\R$} [Bl] at 372 697
\pinlabel {$\R$} [Bl] at 493 697
\pinlabel {$(1)$} at 47 553
\pinlabel {$(2)$} at 186 553
\pinlabel {$(3)$} at 323 553
\pinlabel {$(4)$} at 464 553
\endlabellist
\includegraphics[width=0.7\linewidth]{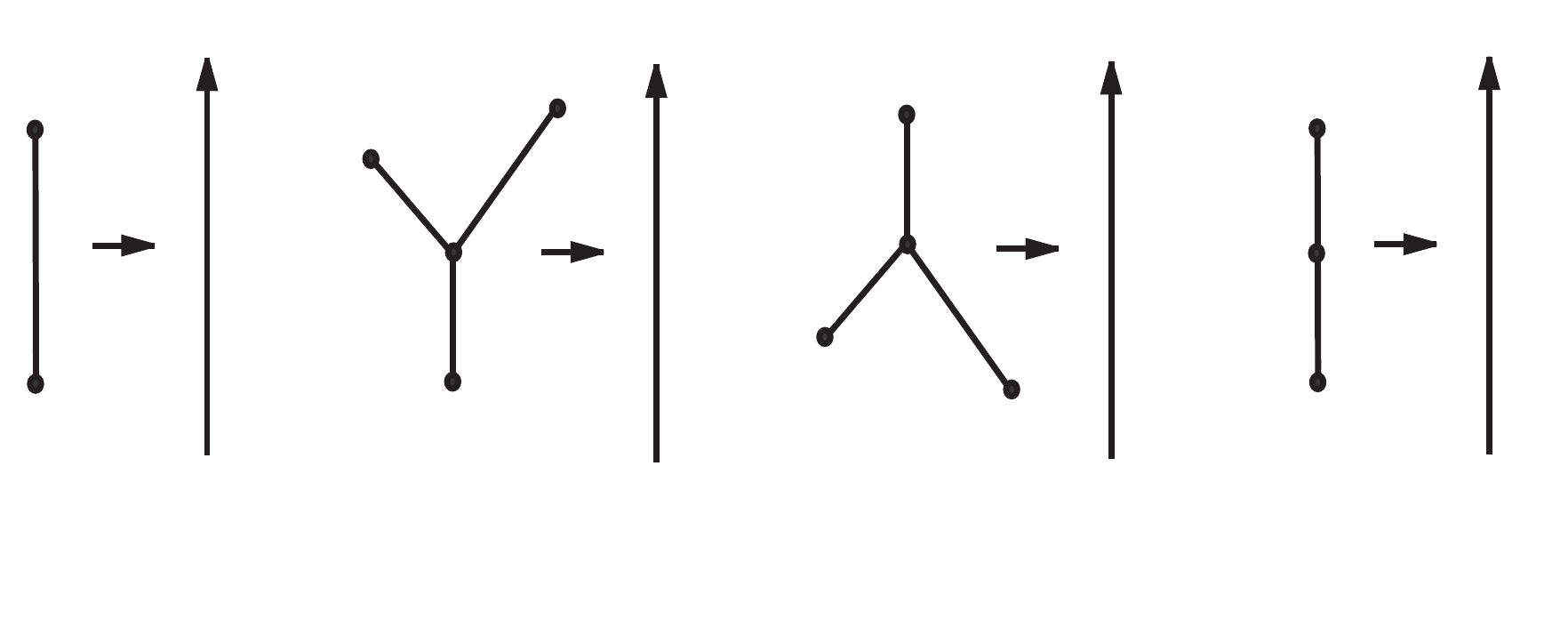}\vspace{-6mm}
\end{center}
\caption{Four elementary functions}
\label{fig3}
\end{figure}
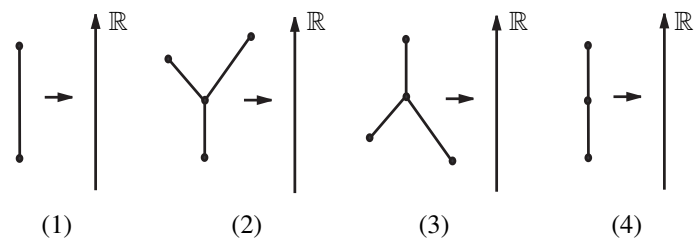

Suppose that $M$ is orientable.
If there is a pair of pieces (2) and
(3), then by the moves (a) and (c) 
(or (a) and (b)) we can replace it
by one or two pieces of type (1).
Furthermore,
by the move (d), we may assume that a
piece as in (1) does not appear. 
Therefore, we end up with an empty graph,
or a disjoint union of several pieces of
type (2) (or a disjoint union of several
pieces of type (3)). Then by the move (a),
we get a standard form.

Suppose now that $M$ is nonorientable.
If there is a pair of pieces of type (4),
then we can replace it with a piece (1)
by the moves (a) and (i). Then the rest of the
argument is the same.

This completes the proof of
\fullref{thm:main1}.

It is easy to observe that in the above proof,
we have only used admissible moves
for simple cobordisms, namely (a)--(d) and (i). 
Therefore, 
the same argument can be applied
to prove \fullref{thm:main2} as well.

\begin{rmk}
In \cite{IS,Kalmar}, the moves (e), (f), (g), (j)
and (k) were used for the proof, and the same argument
cannot be applied to the situation of simple cobordisms.
\end{rmk}

In \fullref{thm:main1} (1) and \fullref{thm:main2} (1),
an isomorphism is given by the map which
associates to each cobordism class of a Morse
function to the sum of the indices $\pm 1$ over
all vertices of degree three of the associated
Reeb function. This sum is equal to
the difference between the numbers of
local maxima and local minima.
In the unoriented cases (\fullref{thm:main1} (2)
and \fullref{thm:main2} (2)),
an isomorphism is constructed by
combining a similar map into $\Z$
and the map which associates to each
cobordism class of a Morse function to
the parity of the number of degree two
vertices of the associated Reeb graph.
Note that this parity coincides with the
parity of the Euler characteristic of the
source surface \cite[Corollary~2.4]{Saeki04}.

\section{Universal complex of singular fibers}\label{section4}

In \cite{Saeki04}, a theory of singular fibers
of differentiable maps has been developed.
The author has introduced the notion of a universal
complex of singular fibers and has
shown that certain cohomology classes of a
universal complex give rise to cobordism invariants
of singular maps.
In this section, we show that the isomorphisms in
Theorems~\ref{thm:main1} and \ref{thm:main2}
can be given by certain cohomology classes
of universal complexes of singular fibers.
This will give explicit examples showing
the effectiveness of the theory of singular fibers
developed in \cite{Saeki04}.

For the terminologies used in this and the
following sections, we refer the reader to
\cite{Saeki04} and also to \cite{Kazarian,Ohmoto,V}.

Let us consider proper $C^\infty$ stable
maps of $3$--manifolds into surfaces.
(Recall that for nice dimensions, a proper
smooth map is $C^\infty$ stable if and only if
it is $C^0$ stable. See \cite{DW}.)
Then we have the list of $C^\infty$ (or
$C^0$) equivalence
classes of singular fibers of such maps
as in \fullref{fig43}. (For the definition
of the $C^\infty$ or $C^0$ equivalence relation for
singular fibers, see \cite[Chapter~1]{Saeki04}.
This can be regarded as the $C^\infty$ or
$C^0$ right-left
equivalence for map germs along the inverse image of
a point.) 
In fact,
every singular fiber of such a map
is $C^\infty$ (or $C^0$)
equivalent to the disjoint union of one of the fibers as in
\fullref{fig43} and a finite number of copies
of a fiber of the trivial circle bundle. 
For details,
see \cite{Saeki04}.

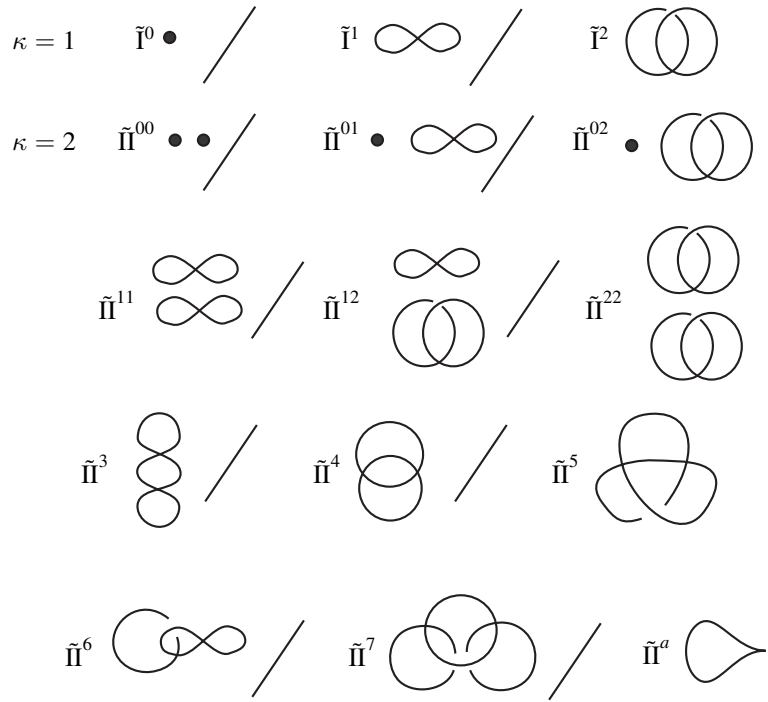
\begin{figure}[ht!]\small
\centering
\labellist\hair 5pt
\pinlabel {$\kappa = 1$} [Br] <-30pt,-2pt> at -7 742
\pinlabel {$\kappa = 2$} [Br] <-30pt,-3pt> at -7 649
\pinlabel {$\tilde{\mathrm{I}}^0$} [Br] <0pt,-2pt> at -7 742
\pinlabel {$\tilde{\mathrm{I}}^1$} [Br] <0pt,-2pt> at 185 742
\pinlabel {$\tilde{\mathrm{I}}^2$} [Br] <0pt,-2pt> at 418 742
\pinlabel {$\tilde{\mathrm{II}}^{00}$} [Br] <0pt,-3pt> at -7 649
\pinlabel {$\tilde{\mathrm{II}}^{01}$} [Br] <0pt,-3pt> at 185 649
\pinlabel {$\tilde{\mathrm{II}}^{02}$} [Br] <0pt,-3pt> at 418 649
\pinlabel {$\tilde{\mathrm{II}}^{11}$} [r] at -25 497
\pinlabel {$\tilde{\mathrm{II}}^{12}$} [r] at 185 497
\pinlabel {$\tilde{\mathrm{II}}^{22}$} [r] <5pt,0pt> at 418 497
\pinlabel {$\tilde{\mathrm{II}}^3$} [r] at -50 343
\pinlabel {$\tilde{\mathrm{II}}^4$} [r] at 167 343
\pinlabel {$\tilde{\mathrm{II}}^5$} [r] at 390 343
\pinlabel {$\tilde{\mathrm{II}}^6$} [r] at -65 171
\pinlabel {$\tilde{\mathrm{II}}^7$} [r] at 200 171
\pinlabel {$\tilde{\mathrm{II}}^a$} [r] at 475 171
\endlabellist
\includegraphics[width=0.7\linewidth]{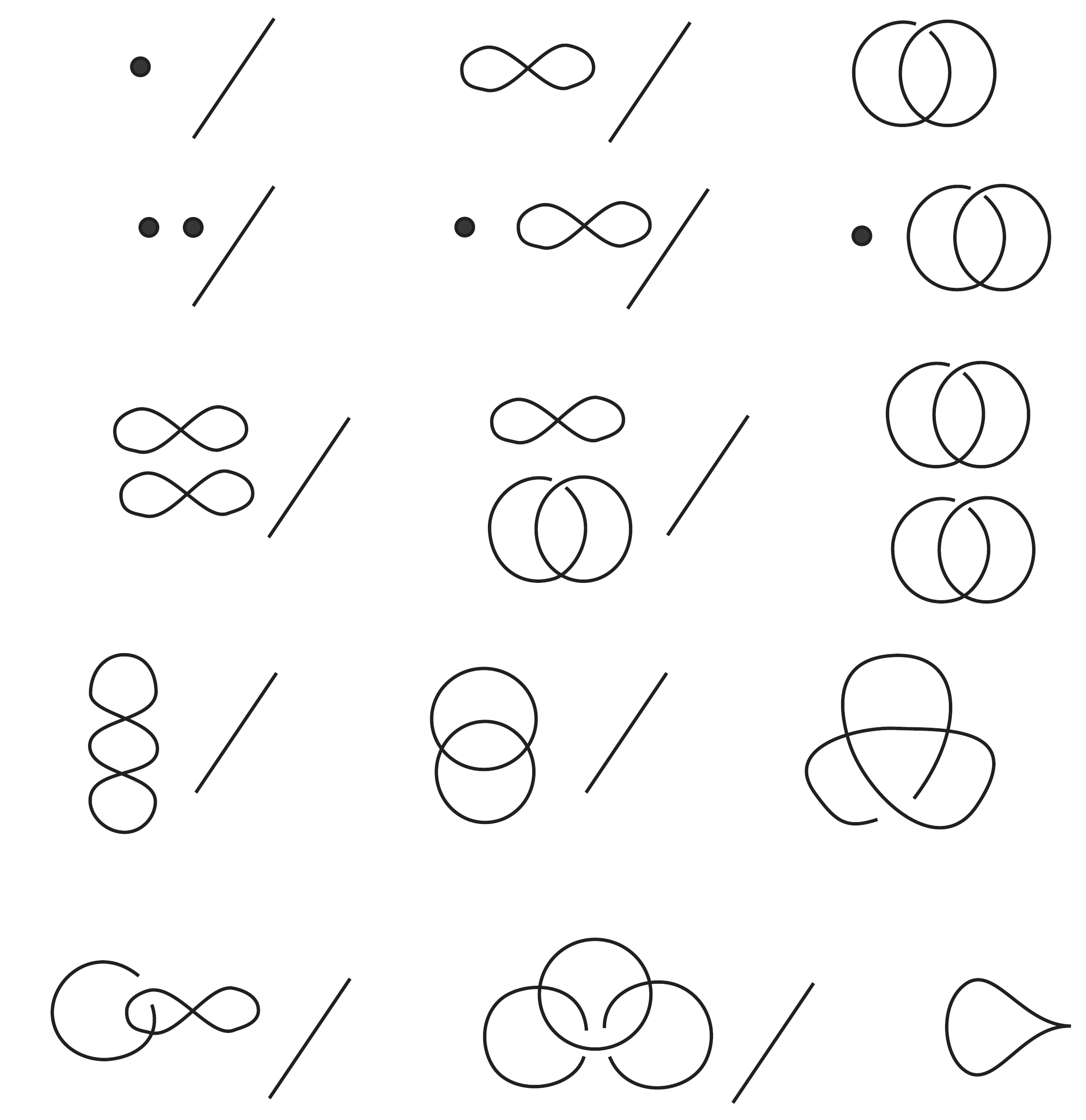}
\caption{List of singular fibers of proper $C^\infty$ stable maps of
$3$--manifolds into surfaces}
\label{fig43}
\end{figure}

Note that in \fullref{fig43},
$\kappa$ denotes the codimension of the set
of points in the target whose corresponding fibers
are equivalent to the relevant one. 
Furthermore, $\tilde{\mathrm{I}}^*$ and
$\tilde{\mathrm{II}}^*$ mean the name of the
corresponding singular fiber, and ``$/$'' is
used only for separating the figures.
The equivalence
class of fibers of codimension zero corresponds to
the class of regular fibers and is unique.
We denote this codimension zero equivalence class
by $\tilde{\mathbf{0}}$. 

We note that the
fiber $\tilde{\mathrm{II}}^a$ corresponds to
a cusp singular point defined as follows.
Let $M$ be a manifold of dimension $n \geq 2$ and
$f \co M \to N$ a smooth map into a surface $N$.
A singular point $x \in S(f)$ of $f$ is called a
\emph{cusp singular point} (or a \emph{cusp point})
if there exist
local coordinates $(x_1, x_2, \ldots, x_n)$ around
$x$ and $(y_1, y_2)$ around $f(x)$ such that
$f$ has the form
$$
y_i \circ f =
\begin{cases}
x_1, \qquad i=1, & \\
x_1x_2+ x_2^3 \pm x_3^2 \pm \cdots \pm x_n^2,
& i = 2.
\end{cases}
$$
If the source $3$--manifold is orientable, then
the singular fibers of types
$\tilde{\mathrm{I}}^2$, $\tilde{\mathrm{II}}^{02}$,
$\tilde{\mathrm{II}}^{12}$, $\tilde{\mathrm{II}}^{22}$,
$\tilde{\mathrm{II}}^5$, $\tilde{\mathrm{II}}^6$ and
$\tilde{\mathrm{II}}^7$ do not appear.

Note also that the list of $C^\infty$
(or $C^0$) equivalence classes of singular fibers
of proper stable Morse functions
on surfaces is nothing but those appearing in
\fullref{fig43} with $\kappa = 1$.

Let $\varrho^0_{n, n-1}(2)$ be the $C^0$ equivalence
relation modulo two circle components for fibers of proper
$C^0$ stable maps of $n$--dimensional
manifolds into $(n-1)$--dimensional manifolds
which are Thom maps.
(For details, see \cite[p.~84]{Saeki04}.
Roughly speaking, two fibers are equivalent with
respect to $\varrho^0_{n, n-1}(2)$ if one is
$C^0$ equivalent to the other one after adding
an even number of regular circle components.)
For a $C^0$ equivalence class $\tilde{\mathfrak{F}}$
of singular fibers, we denote
by $\tilde{\mathfrak{F}}_{\mathrm{o}}$
(or $\tilde{\mathfrak{F}}_{\mathrm{e}}$)
the equivalence class with respect to
$\varrho^0_{n, n-1}(2)$ containing a singular fiber
of type $\tilde{\mathfrak{F}}$ whose total
number of components is odd (resp.\ even).

Let us consider those equivalence classes which
are (strongly) co-orientable in the sense of 
\cite[Definition~10.5]{Saeki04}.
(Roughly speaking, an equivalence class $\tilde{\mathfrak{F}}_\ast$
is strongly co-orientable
if for a given stable map and a point $q$ in the target
whose fiber belongs to $\tilde{\mathfrak{F}}_\ast$,
any local homeomorphism around $q$ preserving 
the adjacent equivalence
classes preserves the orientation of the normal
direction to the submanifold
corresponding to $\tilde{\mathfrak{F}}_\ast$.)
Then we easily get the following for $n = 3$.

\begin{lem}\label{lem:coori}
Those equivalence classes with respect to
$\varrho^0_{3, 2}(2)$ which are
strongly
co-orientable are
$\tilde{\mathbf{0}}_\ast$,
$\tilde{\mathrm{I}}^0_\ast$,
$\tilde{\mathrm{I}}^1_\ast$,
$\tilde{\mathrm{II}}^{01}_\ast$
and $\tilde{\mathrm{II}}^a_\ast$,
where $\ast = \mathrm{o}$ and $\mathrm{e}$.
The other equivalence classes are not
strongly co-orientable.
\end{lem}

The above lemma can be proved by
observing the degenerations of fibers
like those depicted in \cite[Figs.~3.5--3.8]{Saeki04}.

\begin{rmk}
If we consider $\varrho^0_{3, 2}(1)$ 
($C^0$ equivalence modulo regular components)
instead of
$\varrho^0_{3, 2}(2)$, then no
strongly co-orientable equivalence class appears.
For this reason, we have chosen the $C^0$ equivalence
modulo two circle components.
\end{rmk}

Let $\tau$ be the set of singularity types corresponding
to a regular point or a fold point. 
A smooth map between manifolds is called a
\emph{$\tau$--map} if all of its singularities
lie in $\tau$. In other words, a smooth map
is a $\tau$--map if and only if it is a
fold map in the sense of \fullref{section2}.
Let us denote by
$\tau^0(n, p)$ (or $\tau^0(n, p)^\mathrm{ori}$)
the set of all $C^0$--equivalence classes of
fibers for proper $C^0$ stable $\tau$--maps of (orientable)
$n$--dimensional manifolds into $p$--dimensional manifolds
which are Thom maps (for details, see \cite{Saeki04}).
Furthermore, let us denote by $\sigma^0(n,p)$
(or $\sigma^0(n,p)^\mathrm{ori}$) the set of all
$C^0$--equivalence classes of
fibers for proper $C^0$ stable \emph{simple} 
$\tau$--maps of (orientable)
$n$--dimensional manifolds into $p$--dimensional manifolds
which are Thom maps.
Let 
\begin{equation}
\label{eq:co}
\begin{split}
\mathcal{CO}^*(\tau^0(n, n-1), \varrho^0_{n, n-1}(2)), & \quad
\mathcal{CO}^*(\tau^0(n, n-1)^\mathrm{ori}, \varrho^0_{n, n-1}(2)), 
\\
\mathcal{CO}^*(\sigma^0(n, n-1), \varrho^0_{n, n-1}(2)) & \quad
\text{\rm and} \quad 
\mathcal{CO}^*(\sigma^0(n, n-1)^\mathrm{ori}, \varrho^0_{n, n-1}(2))
\end{split}
\end{equation}
be the universal complexes of co-orientable singular fibers 
for the respective classes of maps with respect to the
$C^0$ equivalence modulo two circle components\footnote{In 
\cite{Saeki04}, these cochain complexes are
denoted by using the
symbol ``$\mathcal{CO}$'' 
without ``$\ast$'' as superscripts.
However, in this paper, we intentionally put
the superscripts in order to distinguish them
from the corresponding \emph{chain} complexes introduced
in \fullref{section5}.}.
Note that these complexes are defined over the integers $\Z$.

Then by \fullref{lem:coori}, we see that the following equivalence
classes constitute a basis of the $\kappa$--dimensional
cochain group for all the four cochain complexes in
\eqref{eq:co} with $n = 3$,
where $\ast = \mathrm{o}$ and $\mathrm{e}$:
$$
\tilde{\mathbf{0}}_\ast
\quad (\kappa = 0), \quad
\tilde{\mathrm{I}}^0_\ast, \,
\tilde{\mathrm{I}}^1_\ast \quad
(\kappa = 1), \quad
\tilde{\mathrm{II}}^{01}_\ast \quad
(\kappa = 2).
$$
Note that $\tilde{\mathrm{II}}^a_\ast$ do not appear,
since $\tau$--maps have no cusps.
Note also that for $n = 2$, we have the same
bases for $\kappa \leq 1$.

Let us fix a co-orientation for each of the above
equivalence classes. We choose the co-orientation
for each of the equivalence classes of codimension one
such that the co-orientation points from $\tilde{\mathbf{0}}_{\mathrm{e}}$
to $\tilde{\mathbf{0}}_{\mathrm{o}}$.
Then we see that the
coboundary homomorphism 
is given by the following
formulae (for the definition of the coboundary
homomorphisms, see \cite[Chapters~7 and 8]{Saeki04}):
\begin{equation} \begin{split}
\delta_0(\tilde{\mathbf{0}}_{\mathrm{o}})
& = \tilde{\mathrm{I}}^0_{\mathrm{o}} + 
\tilde{\mathrm{I}}^0_{\mathrm{e}} +
\tilde{\mathrm{I}}^1_{\mathrm{o}} +
\tilde{\mathrm{I}}^1_{\mathrm{e}}, \\
\delta_0(\tilde{\mathbf{0}}_{\mathrm{e}})
& = -\tilde{\mathrm{I}}^0_{\mathrm{o}} - 
\tilde{\mathrm{I}}^0_{\mathrm{e}} -
\tilde{\mathrm{I}}^1_{\mathrm{o}} -
\tilde{\mathrm{I}}^1_{\mathrm{e}}, \\
\delta_1(\tilde{\mathrm{I}}^0_{\mathrm{o}})
& = \tilde{\mathrm{II}}^{01}_{\mathrm{o}} -
\tilde{\mathrm{II}}^{01}_{\mathrm{e}}, \\
\delta_1(\tilde{\mathrm{I}}^0_{\mathrm{e}})
& = \tilde{\mathrm{II}}^{01}_{\mathrm{o}} -
\tilde{\mathrm{II}}^{01}_{\mathrm{e}}, \\
\delta_1(\tilde{\mathrm{I}}^1_{\mathrm{o}})
& = -\tilde{\mathrm{II}}^{01}_{\mathrm{o}} +
\tilde{\mathrm{II}}^{01}_{\mathrm{e}}, \\
\delta_1(\tilde{\mathrm{I}}^1_{\mathrm{e}})
& =  -\tilde{\mathrm{II}}^{01}_{\mathrm{o}} +
\tilde{\mathrm{II}}^{01}_{\mathrm{e}}.
\end{split}
\label{eq:delta}
\end{equation}
In the following, we denote by $[\ast]$ the (co)homology
class represented by the (co)cycle $\ast$.
Then, by a straightforward calculation, we get the following.

\begin{lem}\label{lem:coho}
For the cohomology groups of all the four cochain complexes in
\textup{\eqref{eq:co}} with $n = 3$,
we have
\begin{align*}
H^0 & 
\cong \Z \quad \text{\textup{(}generated by
$[\tilde{\mathbf{0}}_{\mathrm{o}} + 
\tilde{\mathbf{0}}_{\mathrm{e}}]$\textup{)}, \quad and}
\\
H^1 & 
\cong \Z \oplus \Z \quad \text{\textup{(}generated by
$\alpha_1 = -[\tilde{\mathrm{I}}^0_{\mathrm{o}} +
\tilde{\mathrm{I}}^1_{\mathrm{e}}]
= [\tilde{\mathrm{I}}^0_{\mathrm{e}} +
\tilde{\mathrm{I}}^1_{\mathrm{o}}]$,
$\alpha_2
= [-\tilde{\mathrm{I}}^0_{\mathrm{o}} +
\tilde{\mathrm{I}}^0_{\mathrm{e}}]$},
\\
& 
\text{and $\alpha_3 = [\tilde{\mathrm{I}}^1_{\mathrm{o}} -
\tilde{\mathrm{I}}^1_{\mathrm{e}}]$
with $2\alpha_1 = \alpha_2 + \alpha_3$\textup{)}}. 
\end{align*}
Furthermore, for $n = 2$, the same isomorphism
holds for $H^0$, and for $H^1$, we have
\begin{align*}
H^1 &
\cong \Z \oplus \Z \oplus \Z
\quad \text{\textup{(}generated by
$\beta_1 = -[\tilde{\mathrm{I}}^0_{\mathrm{o}} +
\tilde{\mathrm{I}}^1_{\mathrm{e}}]
= [\tilde{\mathrm{I}}^0_{\mathrm{e}} +
\tilde{\mathrm{I}}^1_{\mathrm{o}}]$,
$\beta_2
= [\tilde{\mathrm{I}}^0_{\mathrm{o}}]$},
\\ &
\text{and $\beta_3 = 
[\tilde{\mathrm{I}}^1_{\mathrm{o}}]$\textup{)}}. 
\end{align*}
\end{lem}

\begin{rmk}
The above result for $n = 3$ appears to be different from
\cite[Proposition~14.3]{Saeki04}. In fact, in 
\cite{Saeki04}, opposite co-orientations are
used for $[\tilde{\mathrm{I}}^0_{\mathrm{o}}]$ and
$[\tilde{\mathrm{I}}^1_{\mathrm{e}}]$.
\end{rmk}

Let
$$s^{0 \ast}_\kappa \co H^\kappa(
\mathcal{CO}^*(\tau^0(3, 2), \varrho^0_{3, 2}(2)))
\to
H^\kappa(\mathcal{CO}^*(\tau^0(2, 1), \varrho^0_{2, 1}(2)))$$
etc.\ be the homomorphism induced by suspension\footnote{In
\cite{Saeki04}, the notation $s^0_{\kappa \ast}$ is used
instead of $s^{0 \ast}_\kappa$. However, the latter
should have been used,
since it corresponds to a pull-back. More details will
be explained in \fullref{section5}.}.
Then for $\kappa = 1$, we have
$$s^{0\ast}_1 \alpha_1 = \beta_1, \,
s^{0\ast}_1 \alpha_2 = \beta_1 - \beta_2 - \beta_3
\text{ and } 
s^{0\ast}_1 \alpha_3 = \beta_1 + \beta_2 + \beta_3.$$
In particular, we see that
$s^{0\ast}_1$ is injective and its image
is isomorphic to $\Z \oplus \Z$.

Let $f \co M \to \R$ be an arbitrary stable Morse function
on a closed surface $M$. We give the orientation to $\R$
which points to the increasing direction.
For a $\kappa$--dimensional cohomology class $\alpha$ of the
universal complex of co-orientable
singular fibers represented by
a cocycle $c$, we denote by $\alpha(f) \in H_{1-\kappa}(\R;
\Z)$ the homology class\footnote{For 
$\kappa = 0$ we consider the homology
group with closed support.}
represented by the cycle corresponding to the
closure of the set
of points in $\R$ whose associated
fiber belongs to an equivalence class appearing in $c$
(for details, see \cite[Chapter~11]{Saeki04})\footnote{More
precisely, we consider the chain with
multiplicity given by the corresponding coefficient
in $c$.}.

Then by the same argument
as in the proof of \cite[Lemma~14.1]{Saeki04},
we see that 
$$s^{0\ast}_1 \alpha_1(f) = \beta_1(f)
\in H_0(\R; \Z) \cong \Z$$ 
always vanishes
(see also \fullref{rmk:newproof} of the present
paper).
Furthermore, we have the following.

\begin{lem}
For a Morse function $f$ as above, we have
$$s^{0\ast}_1 \alpha_2(f) 
= -s^{0\ast}_1 \alpha_3(f) = \max{(f)} - \min{(f)}$$
under the natural identification
$H_0(\R; \Z) = \Z$,
where $\max{(f)}$ \textup{(}or $\min{(f)}$\textup{)}
is the number of local maxima \textup{(}resp.\ minima\textup{)}
of the Morse function $f$.
\end{lem}

\begin{proof}
Let $c \in \R$ be a value corresponding to
a local minimum of $f$. If $f^{-1}(c)$
has an odd (or even) number of components, then
it contributes $+1$ (resp.\ $-1$) to 
$||\tilde{\mathrm{I}}^0_{\mathrm{o}}(f)||$
(resp.\ 
$||\tilde{\mathrm{I}}^0_{\mathrm{e}}(f)||$),
where $||\ast||$ refers to the algebraic number
of elements, and for an equivalence
class $\tilde{\mathfrak{F}}$, $\tilde{\mathfrak{F}}(f)$
denotes the set of points in $\R$ over which
lies a singular fiber of type $\tilde{\mathfrak{F}}$.
If $c$ corresponds to a
local maximum, then the signs of contribution
change in both cases. Therefore, we have the
desired conclusion.
\end{proof}

Note that by \cite{Saeki04}, for any $1$--dimensional
cohomology class $\alpha$
of the universal complex as in \eqref{eq:co}
with $n = 3$,
$s^{0\ast}_1 \alpha(f)
\in H_0(\R; \Z) \cong \Z$ gives a fold cobordism
invariant for stable Morse functions $f$ on
closed surfaces.
By the proofs of \fullref{thm:main1} (1) and
\fullref{thm:main2} (1), 
we see that the maps
$$\Phi^{SO} \co \mathcal{M}^{SO}(2) \to \Z \quad
\text{ and } \quad S\Phi^{SO} \co \mathcal{SM}^{SO}(2)
\to \Z$$
which send the cobordism class of a stable
Morse function $f$
to $s^{0\ast}_1 \alpha_2(f) = \max{(f)} - \min{(f)}
\in \Z$ are isomorphisms.

In the unoriented case, the corresponding maps
do not give isomorphisms according to
\fullref{thm:main1} (2) and \fullref{thm:main2} (2).
In order to get isomorphisms, let us consider the
universal complexes of singular fibers
\begin{equation}
\label{eq:co2}
\mathcal{C}^*(\tau^0(n, n-1), \varrho^0_{n, n-1}(2))
\quad \text{ and } \quad
\mathcal{C}^*(\sigma^0(n, n-1), \varrho^0_{n, n-1}(2))
\end{equation}
with coefficients in $\Z_2$. (Here again,
we put ``$\ast$'' as superscripts.)

The coboundary homomorphisms for the
case of $\tau^0(3, 2)$
satisfy the following:
\begin{equation*}
\begin{split}
\delta_0(\tilde{\mathbf{0}}_{\mathrm{o}})
& = \tilde{\mathrm{I}}^0_{\mathrm{o}} + 
\tilde{\mathrm{I}}^0_{\mathrm{e}} +
\tilde{\mathrm{I}}^1_{\mathrm{o}} +
\tilde{\mathrm{I}}^1_{\mathrm{e}}, \\
\delta_0(\tilde{\mathbf{0}}_{\mathrm{e}})
& = \tilde{\mathrm{I}}^0_{\mathrm{o}} +
\tilde{\mathrm{I}}^0_{\mathrm{e}} +
\tilde{\mathrm{I}}^1_{\mathrm{o}} +
\tilde{\mathrm{I}}^1_{\mathrm{e}}, \\
\delta_1(\tilde{\mathrm{I}}^0_{\mathrm{o}})
& = \tilde{\mathrm{II}}^{01}_{\mathrm{o}} +
\tilde{\mathrm{II}}^{01}_{\mathrm{e}}, \\
\delta_1(\tilde{\mathrm{I}}^0_{\mathrm{e}})
& = \tilde{\mathrm{II}}^{01}_{\mathrm{o}} +
\tilde{\mathrm{II}}^{01}_{\mathrm{e}}, \\
\delta_1(\tilde{\mathrm{I}}^1_{\mathrm{o}})
& = \tilde{\mathrm{II}}^{01}_{\mathrm{o}} +
\tilde{\mathrm{II}}^{01}_{\mathrm{e}}, \\
\delta_1(\tilde{\mathrm{I}}^1_{\mathrm{e}})
& = \tilde{\mathrm{II}}^{01}_{\mathrm{o}} +
\tilde{\mathrm{II}}^{01}_{\mathrm{e}}, \\
\delta_1(\tilde{\mathrm{I}}^2_{\mathrm{o}})
& = \tilde{\mathrm{II}}^{02}_{\mathrm{o}} +
\tilde{\mathrm{II}}^{02}_{\mathrm{e}} +
\tilde{\mathrm{II}}^{12}_{\mathrm{o}} +
\tilde{\mathrm{II}}^{12}_{\mathrm{e}} +
\tilde{\mathrm{II}}^6_{\mathrm{o}} +
\tilde{\mathrm{II}}^6_{\mathrm{e}},
\\
\delta_1(\tilde{\mathrm{I}}^2_{\mathrm{e}})
& = \tilde{\mathrm{II}}^{02}_{\mathrm{o}} +
\tilde{\mathrm{II}}^{02}_{\mathrm{e}} +
\tilde{\mathrm{II}}^{12}_{\mathrm{o}} +
\tilde{\mathrm{II}}^{12}_{\mathrm{e}} +
\tilde{\mathrm{II}}^6_{\mathrm{o}} +
\tilde{\mathrm{II}}^6_{\mathrm{e}}.
\end{split}
\end{equation*}
For the case of 
$\sigma^0(3, 2)$, 
we obtain the formulae for the coboundary
homomorphisms by ignoring
$\tilde{\mathrm{II}}^6_\ast$
above. For the cases of $\tau^0(2, 1)$
and $\sigma^0(2, 1)$, the same formulae
hold for $\delta_0$.

By a straightforward calculation, we get the following.

\begin{lem}
For the cohomology groups of the cochain complexes in
\textup{\eqref{eq:co2}} with $n = 3$,
we have
\begin{align*}
H^0 & 
\cong \Z_2 \quad \text{\textup{(}generated by
$[\tilde{\mathbf{0}}_{\mathrm{o}} + 
\tilde{\mathbf{0}}_{\mathrm{e}}]$\textup{)}, \quad and}
\\
H^1 & 
\cong \Z_2 \oplus \Z_2 \oplus \Z_2
\quad \text{\textup{(}generated by
$\hat{\alpha}_1 = [\tilde{\mathrm{I}}^0_{\mathrm{o}} +
\tilde{\mathrm{I}}^1_{\mathrm{e}}]
= [\tilde{\mathrm{I}}^0_{\mathrm{e}} +
\tilde{\mathrm{I}}^1_{\mathrm{o}}]$,} \\
& \text{$\hat{\alpha}_2
= [\tilde{\mathrm{I}}^0_{\mathrm{o}} +
\tilde{\mathrm{I}}^0_{\mathrm{e}}] =
[\tilde{\mathrm{I}}^1_{\mathrm{o}} +
\tilde{\mathrm{I}}^1_{\mathrm{e}}]$
and $\hat{\alpha}_3
= [\tilde{\mathrm{I}}^2_{\mathrm{o}} +
\tilde{\mathrm{I}}^2_{\mathrm{e}}]$\textup{)}}. 
\end{align*}
Furthermore, for $n = 2$, the same isomorphism
holds for $H^0$, and for $H^1$, we have
\begin{align*}
H^1 &
\cong \Z_2 \oplus \Z_2 \oplus \Z_2 \oplus \Z_2
\oplus \Z_2
\quad \text{\textup{(}generated by
$\hat{\beta}_1 = [\tilde{\mathrm{I}}^0_{\mathrm{o}} +
\tilde{\mathrm{I}}^1_{\mathrm{e}}]
= [\tilde{\mathrm{I}}^0_{\mathrm{e}} +
\tilde{\mathrm{I}}^1_{\mathrm{o}}]$,}
\\ &
\text{$\hat{\beta}_2
= [\tilde{\mathrm{I}}^0_{\mathrm{o}}]$,
$\hat{\beta}_3 = 
[\tilde{\mathrm{I}}^1_{\mathrm{o}}]$,
$\hat{\beta}_4 = 
[\tilde{\mathrm{I}}^2_{\mathrm{o}}]$
and $\hat{\beta}_5 = 
[\tilde{\mathrm{I}}^2_{\mathrm{e}}]$\textup{)}}. 
\end{align*}
\end{lem}

We can also describe the homomorphisms
induced by suspension with respect to the
above generators.

Let $f \co M \to \R$ be a stable
Morse function on a closed surface $M$.
Then we see that
$s^{0\ast}_1 \hat{\alpha}_1(f)
\in H_0(\R; \Z_2) \cong \Z_2$
always vanishes as before. Furthermore,
$s^{0\ast}_1 \hat{\alpha}_2(f)$
coincides with $\min{(f)} + \max{(f)}$
modulo two.
Finally,
$s^{0\ast}_1 \hat{\alpha}_3(f)$
gives the number of singular fibers of
type $\tilde{\mathrm{I}}^2$ of $f$.
Therefore, according to the proofs of
\fullref{thm:main1} (2) and
\fullref{thm:main2} (2),
we see that the homomorphisms
$$\Phi \co \mathcal{M}(2) \to \Z \oplus \Z_2 \quad
\text{ and } \quad S\Phi \co \mathcal{SM}(2)
\to \Z \oplus \Z_2$$
which send the cobordism class of a stable
Morse function $f$ to 
$$(s^{0\ast}_1 \alpha_2(f),
s^{0\ast}_1 \hat{\alpha}_3(f)) = 
(\max{(f)} - \min{(f)}, |\tilde{\mathrm{I}}^2(f)|)
\in \Z \oplus \Z_2$$ 
are isomorphisms,
where $|\ast|$ denotes the number of elements modulo two.

Note that by \cite[Corollary~2.4]{Saeki04},
$|\tilde{\mathrm{I}}^2(f)| \in \Z_2$ coincides
with the parity of the Euler characteristic
$\chi(M)$ of the surface $M$.

As the above observations show, the cohomology
classes of universal complexes of singular fibers
can give complete cobordism invariants
for singular maps.

\section{Universal homology complex of singular 
fibers}\label{section5}

In the previous section, we have seen that
cohomology classes of the universal complexes of
singular fibers give
rise to complete cobordism invariants
in our situations.
In order to construct such invariants in the
unoriented case, we had to combine the universal
complex of co-orientable singular fibers and
that of usual singular fibers which are not
necessarily co-orientable. 

In \cite{Kazarian}, Kazarian introduced
the notion of a universal \emph{homology complex}
of singularities, which combines the universal
cohomology complex of co-orientable singularities and
that of usual (not necessarily co-orientable)
singularities, and which is constructed
by reversing the arrows. In this section, we will
pursue the same procedure in our situation
of singular fibers.

Let us consider the case of proper $C^0$ stable fold
maps of (possibly nonorientable)
$n$--dimensional manifolds into 
$(n-1)$--dimensional manifolds.
(The case of simple maps or that of
maps of oriented manifolds can be treated similarly.)
Let
$$
\mathcal{C}_*(\tau^0(n, n-1), \varrho^0_{n, n-1}(2))
$$
be the chain complex defined as follows.
For each $\kappa$, the $\kappa$--dimensional
chain group, denoted by $C_\kappa(\tau^0(n, n-1), 
\varrho^0_{n, n-1}(2))$,
is the direct sum, over all equivalence classes
of singular fibers of codimension $\kappa$
with respect to $\varrho^0_{n, n-1}(2)$,
of the groups $\Z$ for co-orientable classes and
the groups $\Z_2$ for non co-orientable classes, and we
denote the generators by using the same
symbols for the corresponding equivalence classes of singular fibers.

Let $\mathfrak{F}$ and $\mathfrak{G}$ be two equivalence
classes of singular fibers such that $\kappa(\mathfrak{F})
= \kappa(\mathfrak{G}) + 1$, where $\kappa$ denotes the
codimension.
Let $f \co M \to N$ be a proper $C^0$ stable fold map
of an $n$--dimensional manifold into an $(n-1)$--dimensional
manifold which is a Thom map. 
Let us denote by $\mathfrak{F}(f)$ (or
$\mathfrak{G}(f)$) the set of points in $N$
over which lies a singular fiber of type $\mathfrak{F}$
(resp.\ $\mathfrak{G}$). Note that $\mathfrak{F}(f)$
and $\mathfrak{G}(f)$ are submanifolds
of $N$ of codimensions $\kappa(\mathfrak{F})$ and
$\kappa(\mathfrak{G})$ respectively.
Let us consider a point $q \in \mathfrak{F}(f)$
and a small disk $D_q$ of dimension $\kappa(\mathfrak{F})$
centered at $q$ which intersects $\mathfrak{F}(f)$
transversely exactly at $q$.
Then $\mathfrak{G}(f)$ cuts $D_q$
in a finite set of curves. If $\mathfrak{G}$ is not
co-orientable, then we define $[\mathfrak{G}:
\mathfrak{F}] \in \Z_2$ as the parity of the
number of these curves. If
$\mathfrak{G}$ is co-orientable, then the chosen
co-orientation of $\mathfrak{G}$ together with
the chosen orientation of $D_q$
allows us to define a sign for each
curve. We define $[\mathfrak{G}:
\mathfrak{F}] \in \Z$ as the algebraic
number of these curves, counted with signs.
Note that if $\mathfrak{G}$ is co-orientable and
$\mathfrak{F}$ is \emph{not} co-orientable,
then we always have $[\mathfrak{G}: \mathfrak{F}]
= 0$. Note also that the incidence coefficient
$[\mathfrak{G}: \mathfrak{F}]$ thus defined does not depend on
the choice of $q$ etc.\ and is well-defined for all
the above cases.

Now the boundary homomorphism
$$\partial_\kappa \co C_\kappa(\tau^0(n, n-1), \varrho^0_{n, n-1}(2))
\to C_{\kappa-1}(\tau^0(n, n-1), \varrho^0_{n, n-1}(2))$$
is defined by the formula
$$\partial_\kappa(\mathfrak{F}) = \sum_{\kappa(\mathfrak{G})
= \kappa(\mathfrak{F}) - 1}[\mathfrak{G}:
\mathfrak{F}] \mathfrak{G}$$
for the generators $\mathfrak{F}$ of
$C_\kappa(\tau^0(n, n-1), \varrho^0_{n, n-1}(2))$.
Note that this is a well-defined homomorphism.

It is easy to check that $\partial_{\kappa-1}
\circ \partial_\kappa = 0$ as in 
\cite{Saeki04,V}.
The chain complex
\begin{equation}
\mathcal{C}_*(\tau^0(n, n-1), \varrho^0_{n, n-1}(2))
= (C_\kappa(\tau^0(n, n-1), \varrho^0_{n, n-1}(2)),
\partial_\kappa)_\kappa
\label{eq:c*}
\end{equation}
thus constructed is called the
\emph{universal homology complex of singular
fibers} for $C^0$ stable fold maps of
$n$--dimensional manifolds into $(n-1)$--dimensional
manifolds.

\begin{rmk}
In the definition of the universal
complex given in \cite{Saeki04},
we have formally allowed infinite sums
as elements of the cochain groups. However,
for the universal homology complex that
we have defined here, we consider the
\emph{direct sum} of some copies of
$\Z$ and $\Z_2$,
and we do not allow infinite sums.
Therefore, the boundary homomorphism
is well-defined.
\end{rmk}

As in \cite{Kazarian}, we can check that
the universal cochain complex of singular fibers
$$\mathcal{C}^*(\tau^0(n, n-1), \varrho^0_{n, n-1}(2))$$
and the universal cochain complex of co-orientable
singular fibers
$$\mathcal{CO}^*(\tau^0(n, n-1), \varrho^0_{n, n-1}(2))$$
as defined in \cite{Saeki04}
are isomorphic to
$$\mathrm{Hom}(\mathcal{C}_*(\tau^0(n, n-1), 
\varrho^0_{n, n-1}(2)), \Z_2)$$
and
$$\mathrm{Hom}(\mathcal{C}_*(\tau^0(n, n-1), 
\varrho^0_{n, n-1}(2)), \Z)$$
respectively.
In this sense, the universal homology complex
\eqref{eq:c*} unifies the universal complex
of usual singular fibers with coefficients
in $\Z_2$ and that of co-orientable ones
with coefficients in $\Z$.

Let us proceed to the explicit calculation
in the case of $n = 3$.
The generators of $C_\kappa(\tau^0(3, 2), \varrho^0_{3, 2}(2))$
are as given in Table~\ref{table1} (see 
also \fullref{fig43}).

\begin{table}[b]
\centering
\caption{Generators of $C_\kappa(\tau^0(3, 2), \varrho^0_{3, 2}(2))$}
\label{table1}
\begin{tabular}{|c||c|c|}
\hline 
$\kappa$ & $\Z$ & $\Z_2$ \\ \hline \hline
$0$ & \rule[0.5cm]{0cm}{0cm}$\tilde{\mathbf{0}}_{\mathrm{o}},
\tilde{\mathbf{0}}_{\mathrm{e}}$ & \\ \hline
$1$ & \rule[0.5cm]{0cm}{0cm}$\tilde{\mathrm{I}}^0_{\mathrm{o}},
\tilde{\mathrm{I}}^0_{\mathrm{e}},
\tilde{\mathrm{I}}^1_{\mathrm{o}},
\tilde{\mathrm{I}}^1_{\mathrm{e}}$ &
$\tilde{\mathrm{I}}^2_{\mathrm{o}},
\tilde{\mathrm{I}}^2_{\mathrm{e}}$ \\ \hline
$2$ & 
\rule[0.5cm]{0cm}{0cm}$\tilde{\mathrm{II}}^{01}_{\mathrm{o}},
\tilde{\mathrm{II}}^{01}_{\mathrm{e}}$
&
$\tilde{\mathrm{II}}^{00}_{\mathrm{o}},
\tilde{\mathrm{II}}^{00}_{\mathrm{e}},
\tilde{\mathrm{II}}^{11}_{\mathrm{o}},
\tilde{\mathrm{II}}^{11}_{\mathrm{e}},
\tilde{\mathrm{II}}^{02}_{\mathrm{o}},
\tilde{\mathrm{II}}^{02}_{\mathrm{e}},
\tilde{\mathrm{II}}^{12}_{\mathrm{o}},
\tilde{\mathrm{II}}^{12}_{\mathrm{e}},
\tilde{\mathrm{II}}^{22}_{\mathrm{o}},
\tilde{\mathrm{II}}^{22}_{\mathrm{e}}$, \\
& & 
$\tilde{\mathrm{II}}^3_{\mathrm{o}},
\tilde{\mathrm{II}}^3_{\mathrm{e}},
\tilde{\mathrm{II}}^4_{\mathrm{o}},
\tilde{\mathrm{II}}^4_{\mathrm{e}},
\tilde{\mathrm{II}}^5_{\mathrm{o}},
\tilde{\mathrm{II}}^5_{\mathrm{e}},
\tilde{\mathrm{II}}^6_{\mathrm{o}},
\tilde{\mathrm{II}}^6_{\mathrm{e}},
\tilde{\mathrm{II}}^7_{\mathrm{o}},
\tilde{\mathrm{II}}^7_{\mathrm{e}}$
\\ \hline
\end{tabular}
\end{table}

The boundary homomorphisms are given as follows.
\begin{equation}
\label{eq:boundary}
\begin{split}
&
\tilde{\mathrm{II}}^{01}_{\mathrm{o}} \mapsto
\tilde{\mathrm{I}}^0_{\mathrm{o}} +
\tilde{\mathrm{I}}^0_{\mathrm{e}} -
\tilde{\mathrm{I}}^1_{\mathrm{o}} -
\tilde{\mathrm{I}}^1_{\mathrm{e}}, \\ &
\tilde{\mathrm{II}}^{01}_{\mathrm{e}} \mapsto
-\tilde{\mathrm{I}}^0_{\mathrm{o}} - 
\tilde{\mathrm{I}}^0_{\mathrm{e}} +
\tilde{\mathrm{I}}^1_{\mathrm{o}} +
\tilde{\mathrm{I}}^1_{\mathrm{e}}, \\ &
\tilde{\mathrm{II}}^{00}_{\mathrm{o}},
\tilde{\mathrm{II}}^{00}_{\mathrm{e}},
\tilde{\mathrm{II}}^{11}_{\mathrm{o}},
\tilde{\mathrm{II}}^{11}_{\mathrm{e}},
\tilde{\mathrm{II}}^{22}_{\mathrm{o}},
\tilde{\mathrm{II}}^{22}_{\mathrm{e}},
\tilde{\mathrm{II}}^3_{\mathrm{o}},
\tilde{\mathrm{II}}^3_{\mathrm{e}},
\tilde{\mathrm{II}}^4_{\mathrm{o}},
\tilde{\mathrm{II}}^4_{\mathrm{e}},
\tilde{\mathrm{II}}^5_{\mathrm{o}},
\tilde{\mathrm{II}}^5_{\mathrm{e}},
\tilde{\mathrm{II}}^7_{\mathrm{o}},
\tilde{\mathrm{II}}^7_{\mathrm{e}}
\mapsto 0, \\ &
\tilde{\mathrm{II}}^{02}_{\mathrm{o}},
\tilde{\mathrm{II}}^{02}_{\mathrm{e}},
\tilde{\mathrm{II}}^{12}_{\mathrm{o}},
\tilde{\mathrm{II}}^{12}_{\mathrm{e}},
\tilde{\mathrm{II}}^6_{\mathrm{o}},
\tilde{\mathrm{II}}^6_{\mathrm{e}}
\mapsto
\tilde{\mathrm{I}}^2_{\mathrm{o}} +
\tilde{\mathrm{I}}^2_{\mathrm{e}}, \\ &
\tilde{\mathrm{I}}^0_{\mathrm{o}},
\tilde{\mathrm{I}}^0_{\mathrm{e}},
\tilde{\mathrm{I}}^1_{\mathrm{o}},
\tilde{\mathrm{I}}^1_{\mathrm{e}}
\mapsto 
\tilde{\mathbf{0}}_{\mathrm{o}} - 
\tilde{\mathbf{0}}_{\mathrm{e}}, \\ &
\tilde{\mathrm{I}}^2_{\mathrm{o}},
\tilde{\mathrm{I}}^2_{\mathrm{e}}
\mapsto 0.
\end{split}
\end{equation}
Then a straightforward calculation 
shows the following.

\begin{lem}
For the homology groups of the chain complex
$$\mathcal{C}_*(\tau^0(3, 2), \varrho^0_{3, 2}(2)),$$
we have
\begin{align*}
H_0 & 
\cong \Z \quad \text{\textup{(}generated by
$[\tilde{\mathbf{0}}_{\mathrm{o}}] = 
[\tilde{\mathbf{0}}_{\mathrm{e}}]$\textup{)}, \quad and}
\\
H_1 & 
\cong 
\Z \oplus \Z \oplus \Z_2
\quad \text{\textup{(}generated by
$\tilde{\alpha}_1 = [\tilde{\mathrm{I}}^0_{\mathrm{o}} -
\tilde{\mathrm{I}}^1_{\mathrm{e}}]
= [\tilde{\mathrm{I}}^1_{\mathrm{o}} -
\tilde{\mathrm{I}}^0_{\mathrm{e}}]$,} \\
& \text{$\tilde{\alpha}_2
= [-\tilde{\mathrm{I}}^0_{\mathrm{o}} +
\tilde{\mathrm{I}}^0_{\mathrm{e}}]$
and $\tilde{\alpha}_3
= [\tilde{\mathrm{I}}^2_{\mathrm{o}}] =
[\tilde{\mathrm{I}}^2_{\mathrm{e}}]$\textup{)}}. 
\end{align*}
\end{lem}

Note that for $H_1$, we can replace $\tilde{\alpha}_2$ by
$$\tilde{\alpha}'_2 = [-\tilde{\mathrm{I}}^1_{\mathrm{o}} +
\tilde{\mathrm{I}}^1_{\mathrm{e}}],$$
since we have the relation $-2\tilde{\alpha}_1
= \tilde{\alpha}_2 + \tilde{\alpha}'_2$.

In order to consider the hypercohomologies, 
in the sense of \cite{Kazarian}, 
of the universal homology complex constructed above, 
let us consider a free
approximation $\mathcal{F}$ of
$$\mathcal{V} = 
\mathcal{C}_*(\tau^0(3, 2), \varrho^0_{3, 2}(2))$$
(for the definition of a free approximation,
see \cite[Chapter 5, Section 2]{Spanier}
or \cite{Kazarian}). For each $\kappa$, let us denote the
$\kappa$--dimensional cochain group of $\mathcal{F}$
by $F_\kappa$. Then
the generators of the free abelian
group $F_\kappa$, $\kappa = 0, 1, 2$,
are given by the elements corresponding
to those of $\mathcal{V}$, except that
we need to add one generator $A$ to $F_2$.
We denote the corresponding generators by the same
symbols as in Table~\ref{table1}.
The boundary homomorphism $\partial_\kappa
\co F_\kappa \to F_{\kappa-1}$ is given
by the same formulae as in \eqref{eq:boundary}
and by 
$$\partial_2(A) = 2 \,
\tilde{\mathrm{I}}^2_{\mathrm{o}}.$$
Furthermore, the epimorphism
$\lambda \co \mathcal{F} \to \mathcal{V}$
is naturally defined by the obvious
correspondence together with
$\lambda(A) = 0$.
It is straightforward to check that
$\lambda$ is a chain map and the induced homomorphism
$\lambda_* \co H_*(\mathcal{F}; \Z) 
\to H_*(\mathcal{V}; \Z)$ is
an isomorphism.

For an abelian group $G$,
the \emph{hypercohomology} $\mathbb{H}^*(\mathcal{V}; G)$
of $\mathcal{V}$
with coefficients in $G$ is, by definition,
$H^*(\mathcal{F}; G)$. This is well-defined and
depends only on $\mathcal{V}$ and $G$
(for details, see \cite{Spanier}).

Recall the canonical isomorphisms:
\begin{equation*}
\begin{split}
\mathcal{C}^*(\tau^0(3, 2), \varrho^0_{3, 2}(2))
& \cong \mathrm{Hom}(\mathcal{V}, \Z_2) \quad \text{ and} \\
\mathcal{CO}^*(\tau^0(3, 2), \varrho^0_{3, 2}(2))
& \cong \mathrm{Hom}(\mathcal{V}, \Z).
\end{split}
\end{equation*}
Then a straightforward calculation shows
the following.

\begin{lem}\label{lem:isom}
The following homomorphisms induced by $\lambda$
are both isomorphisms for $\kappa = 0, 1$:
\begin{multline*}
H^\kappa(\mathcal{C}^*(\tau^0(3, 2), \varrho^0_{3, 2}(2))
= H^\kappa(\mathcal{V}; \Z_2)
\\ \to H^\kappa(\mathcal{F}; \Z_2) 
= \mathbb{H}^\kappa(\mathcal{C}_*(\tau^0(3, 2), 
\varrho^0_{3, 2}(2)); \Z_2), 
\end{multline*}
\begin{multline*}
H^\kappa(\mathcal{CO}^*(\tau^0(3, 2), \varrho^0_{3, 2}(2))
= H^\kappa(\mathcal{V}; \Z)
\\ \to H^\kappa(\mathcal{F}; \Z)
= \mathbb{H}^\kappa(\mathcal{C}_*(\tau^0(3, 2), 
\varrho^0_{3, 2}(2)); \Z).
\end{multline*}
\end{lem}

\begin{rmk}
As the above lemma shows, for $n = 3$ the homomorphisms
$\lambda^*$ are isomorphisms. However, for
$n > 3$, we do not know if this is true or not.
\end{rmk}

Let $f \co M \to N$ be a proper $C^0$ stable
$\tau$--map of an $n$--dimensional manifold
into an $(n-1)$--dimensional manifold
which is a Thom map.
Then by the same argument as in \cite[\S5]{Kazarian},
we can define a natural homomorphism
$$\tilde{\varphi}_f^* \co \mathbb{H}^*(\mathcal{V}(n, n-1); G)
\to H^*(N; G)$$
in such a way that for $G = \Z$ and $\Z_2$, we have
$$\varphi_f^* = \tilde{\varphi}_f^* \circ \lambda^*
\co H^*(\mathcal{V}(n, n-1); G) \to H^*(N; G),$$
where $\varphi^*_f$ refers to the
homomorphism induced by $f$
in the sense of \cite[Chapter~11]{Saeki04}
and 
$$\mathcal{V}(n, n-1) = \mathcal{C}_*
(\tau^0(n, n-1), \varrho^0_{n, n-1}(2)).$$
(We have added ``$\ast$'' as superscript
for $\varphi^*_f$,
which will be necessary in the following
argument.) Note also that we can define the
natural homomorphism
$$\varphi_{f\ast} = \tilde{\varphi}_{f\ast} \co H_*(N; G) \to 
\mathbb{H}_*(\mathcal{V}(n, n-1); G)
= H_*(\mathcal{V}(n, n-1); G)$$
for any abelian group $G$. (In fact,
$\varphi_f$ and $\tilde{\varphi}_f$ are
defined on the chain level.)

Let us show that the homomorphism 
$\tilde{\varphi}^*_f$ induced by $f$ defines
a $\tau$--cobordism invariant
of $f$.

By virtue of the uniqueness of the
lift (up to chain homotopy)
for free approximations
in the sense of \cite[Chapter 2, Section 2, Lemma~13]{Spanier}
and \cite[Proposition~2.2]{Kazarian},
we can define the suspension homomorphism
for free approximations of the universal
homology complexes of singular fibers.
More precisely, we have a natural
chain map
$$s \co \mathcal{C}_*(\tau^0(n, n-1), \varrho^0_{n, n-1}(2))
\to \mathcal{C}_*(\tau^0(n+1, n), \varrho^0_{n+1, n}(2))$$
induced by suspension.
(Recall that the suspension
of a map $f \co M \to N$ refers to the map
$f \times \mathrm{id}_\R \co M \times \R \to N \times \R$
and this naturally induces the notion of a suspension
for singular fibers. For details, see
\cite[Definition~8.4]{Saeki04}.)
Then we have a chain map
$\tilde{s} \co \mathcal{F}(n, n-1) \to \mathcal{F}(n+1, n)$,
unique up to chain homotopy,
which makes the following diagram
commutative:
$$
\begin{CD}
\mathcal{F}(n, n-1) @> \tilde{s} >> \mathcal{F}(n+1, n) \\
@V \lambda_n VV  @VV \lambda_{n+1} V \\
\mathcal{C}_*(\tau^0(n, n-1), \varrho^0_{n, n-1}(2))
@> s >> \mathcal{C}_*(\tau^0(n+1, n), \varrho^0_{n+1, n}(2)),
\end{CD}
$$
where $\mathcal{F}(m, m-1)$ denotes
a free approximation of
$\mathcal{C}_*(\tau^0(m, m-1), \varrho^0_{m, m-1}(2))$
and $\lambda_m$ is the corresponding epimorphism
for $m = n, n+1$. Then $\tilde{s}$ induces
the homomorphism
$$\tilde{s}_\kappa^* \co \mathbb{H}^\kappa(\mathcal{V}(n+1, n); G)
\to \mathbb{H}^\kappa(\mathcal{V}(n, n-1); G)$$
for any coefficient abelian group $G$.

Then we have the following.

\begin{prop}\label{prop:inv}
Let $f_i \co M_i \to N$, $i = 0, 1$,
be $C^0$ stable $\tau$--maps of $n$--dimensional
manifolds into an $(n-1)$--dimensional
manifold $N$ which is a Thom map, where we assume that
$M_i$ are closed. If they are $\tau$--cobordant,
then for every $\kappa$, we have
$$\tilde{\varphi}^*_{f_0} \circ
\tilde{s}_{\kappa}^* =
\tilde{\varphi}^*_{f_1} \circ
\tilde{s}_{\kappa}^* \co
\mathbb{H}^\kappa(\mathcal{V}(n+1, n); G)
\to H^\kappa(N; G)$$
for any coefficient abelian group $G$.
In other words, we have
$$\tilde{\varphi}^*_{f_0}|_{\mathrm{Im}\, 
\tilde{s}_{\kappa}^*} = 
\tilde{\varphi}^*_{f_1}|_{\mathrm{Im}\, 
\tilde{s}_{\kappa}^*} \co \mathrm{Im}\, 
\tilde{s}_{\kappa}^* \to H^\kappa(N; G).$$
\end{prop}

\begin{proof}
Let $F \co W \to N \times [0, 1]$ be a $\tau$--cobordism
between $f_0$ and $f_1$. Let us fix cell complex
structures on $N$ and $N \times [0, 1]$ which
are compatible with each other.
We denote by $C(N)$ and $C(N \times [0, 1])$
the chain complexes (over the
integers) of $N$ and $N
\times [0, 1]$ respectively associated with their
cell complex structures.
Then as in \cite{Kazarian} we can
construct chain maps
\begin{equation*}
\begin{split}
& \varphi_{f_j} \co C(N) \to \mathcal{V}(n, n-1), \quad j = 0, 1, 
\quad \text{ and } \\
& \varphi_F \co C(N \times [0, 1]) \to \mathcal{V}(n+1, n).
\end{split}
\end{equation*}
(For this, we do not need any free
approximation.)
Note that then there exist
chain maps
\begin{equation*}
\begin{split}
& \tilde{\varphi}_{f_j} \co C(N) \to \mathcal{F}(n, n-1), \quad j = 0, 1, 
\quad \text{ and } \\
& \tilde{\varphi}_F \co C(N \times [0, 1]) \to \mathcal{F}(n+1, n),
\end{split}
\end{equation*}
unique up to chain homotopy,
such that $\varphi_{f_j} = \lambda_n \circ
\tilde{\varphi}_{f_j}$, $j = 0, 1$, and
$\varphi_F = \lambda_{n+1} \circ \tilde{\varphi}_F$.

Now, let us consider the diagram
of chain complexes as in \fullref{diagram},
where $i_{j\sharp}$ is the chain map induced by
$i_j \co N \to N \times [0, 1]$
defined by $i_j(x) = (x, j)$, $j = 0, 1$.

\begin{figure}[ht!]
\begin{center}
\setlength{\unitlength}{.75mm}
\begin{picture}(100,50)(0,0)
\put(-3,45){\small$\mathcal{F}(n, n-1)$}
\put(95,45){\small$\mathcal{F}(n+1, n)$}
\put(-3,0){\small$\mathcal{V}(n, n-1)$}
\put(95,0){\small$\mathcal{V}(n+1, n)$}
\put(20,46){\vector(1,0){73}}
\put(50,47){$\scriptstyle\tilde{s}$}
\put(10,40){\vector(0,-1){35}}
\put(5,22){$\scriptstyle\lambda_n$}
\put(100,40){\vector(0,-1){35}}
\put(102,22){$\scriptstyle\lambda_{n+1}$}
\put(20,1){\vector(1,0){73}}
\put(50,2){$\scriptstyle s$}
\put(24,22){\small$C(N)$}
\put(37,23){\vector(1,0){27}}
\put(50,25){$\scriptstyle i_{j \sharp}$}
\put(67,22){\small$C(N \times [0, 1])$}
\put(24,30){\vector(-1,1){10}}
\put(19,36){$\scriptstyle \tilde{\varphi}_{f_j}$}
\put(82,30){\vector(1,1){10}}
\put(83,36){$\scriptstyle \tilde{\varphi}_F$}
\put(24,17){\vector(-1,-1){10}}
\put(14,14){$\scriptstyle \varphi_{f_j}$}
\put(82,17){\vector(1,-1){10}}
\put(86,14){$\scriptstyle \varphi_F$}
\end{picture}
\end{center}
\caption{Diagram of chain complexes}
\label{diagram}
\end{figure}
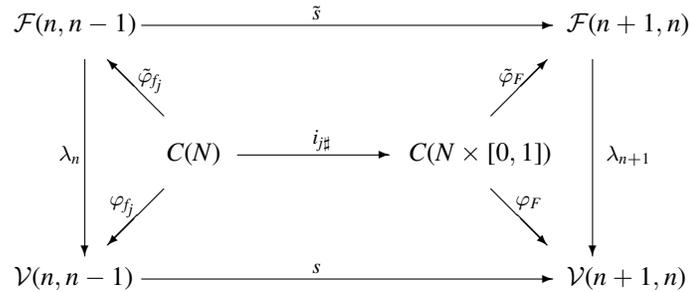

Note that we have
\begin{equation*}
\begin{split}
s \circ \lambda_n & = \lambda_{n+1} \circ \tilde{s}
\quad \text{ and}
\\
s \circ \varphi_{f_j} & = \varphi_F \circ i_{j\sharp},
\quad j = 0, 1,
\end{split}
\end{equation*}
where the latter equality follows from
the constructions of $\varphi_{f_j}$,
$\varphi_F$ and $s$. Therefore, we see that
$$\lambda_{n+1} \circ (\tilde{s} \circ \tilde{\varphi}_{f_j})
= \lambda_{n+1} \circ \tilde{\varphi}_F \circ i_{j\sharp}$$
holds for $j = 0, 1$.
Then by the uniqueness of lifts up to chain homotopy
(see \cite[Chapter 2, Section 2, Lemma~13]{Spanier}
and \cite[Proposition~2.2]{Kazarian}),
we see that the chain maps
$\tilde{s} \circ \tilde{\varphi}_{f_j}$ and
$\tilde{\varphi}_F \circ i_{j\sharp}$ are chain
homotopic. 

Hence, they induce the identical homomorphisms
in homology and cohomology. In particular, we have
the commutative diagram
$$
\begin{CD}
H^*(\mathcal{F}(n,n-1); G)
@< \tilde{s}^\ast
<< H^*(\mathcal{F}(n+1,n); G)
\\
@V \tilde{\varphi}_{f_j}^* VV @VV 
\tilde{\varphi}_F^* V \\
H^*(N; G) @< i_{j*} <<
H^*(N \times [0, 1]; G)
\end{CD}
$$
for any abelian group $G$.
Since $i_0$ and $i_1$
are homotopic, we have $i_0^* = i_1^*$, and the
result follows immediately.
\end{proof}

Thus, it is expected that
the hypercohomology of the
universal homology complex of singular fibers
gives more cobordism invariants
than the usual cohomology of the
universal cochain complex. However, in our situation,
by virtue of \fullref{lem:isom},
we get the same cobordism invariants.

It would be interesting to
study the hypercohomologies of the
higher dimensional analogues
to see if there is an essential difference
between the hypercohomologies and the
usual cohomologies.
If there is, then the corresponding
cobordism invariant would lead to
a ``hidden singular fiber''
in a sense similar to \cite{Kazarian}.

\begin{rmk}
In this section, we have considered only
$\tau$--maps (i.e.\ fold maps). However, this is not
essential, and the same theory
as in this section holds in the
general framework as in \cite{Saeki04}.
\end{rmk}

\section{A topological invariant
for map germs}\label{section6}

In this section, as an application
of the theory of universal
complexes of singular fibers as described
in Sections \ref{section4} and \ref{section5},
we give a new topological invariant
for generic smooth map germs
$(\R^3, 0) \to (\R^2, 0)$.

In what follows, we will not
distinguish a map germ from
its representative when there is no
confusion.

\begin{dfn}\label{dfn:generic}
We say that a smooth
map germ $g \co (\R^3, 0) \to (\R^2, 0)$
is \emph{generic} if
for any sufficiently small positive 
real numbers $\varepsilon$ and $\delta$,
the upper bound of $\delta$ depending on $g$ 
and the upper bound of $\varepsilon$
depending on $\delta$ and $g$, we have
\begin{itemize}
\item[(G1)] $D^3_\delta \cap g^{-1}(S^1_\varepsilon)$
is a smooth manifold possibly with boundary,
\item[(G2)] $g_\partial = g|_{D^3_\delta \cap g^{-1}(S^1_\varepsilon)} \co
D^3_\delta \cap g^{-1}(S^1_\varepsilon) \to
S^1_\varepsilon$
is $C^\infty$ stable,
\item[(G3)] $g|_{\partial D^3_\delta
\cap g^{-1}(D^2_\varepsilon)} \co
\partial D^3_\delta
\cap g^{-1}(D^2_\varepsilon) \to D^2_\varepsilon$
is a submersion, and
\item[(G4)] the restriction
$$g|_{D^3_\delta \cap g^{-1}(D^2_\varepsilon \setminus
\{0\})} \co D^3_\delta \cap g^{-1}(D^2_\varepsilon \setminus
\{0\})
\to D^2_\varepsilon \setminus \{0\}$$ 
is proper,
$C^\infty$ stable and $C^\infty$ equivalent
to the product map 
$$g_\partial \times \id_{(0,
\varepsilon)} \co (D^3_\delta \cap g^{-1}(S^1_\varepsilon))
\times (0, \varepsilon) \to S^1_\varepsilon \times (0, 
\varepsilon)$$ 
defined by $(x, t) \mapsto
(g(x), t)$,
\end{itemize}
where $D^3_\delta$ (or $D^2_\varepsilon$)
denotes the
$3$--dimensional ball in $\R^3$ (resp.\ 
$2$--dimensional disk in $\R^2$) with radius
$\delta$ (resp.\ $\varepsilon$)
centered at the origin.

Note that the set of non-generic
map germs has infinite codimension in
an appropriate sense. For details, see 
the results of Fukuda \cite{Fukuda1}
or Nishimura \cite{Nishimura}.
\end{dfn}

For a generic smooth map germ $g \co (\R^3, 0)
\to (\R^2, 0)$, set $U = D^3_\delta \cap
g^{-1}(\Int{D^2_\varepsilon})$ for
$\delta$ and $\varepsilon$ as above.
Note that $g|_U \co U \to \Int{D^2}$ is a
proper smooth map.
Let $\tilde{g} \co U \to \Int{D^2} \subset \R^2$
be a proper $C^\infty$ stable perturbation of $g|_U$
in a sense similar to that
in \cite{FI,FBS,BS,Ohsumi}.

For a cusp point $x \in U$ of $\tilde{g}$,
we define its sign $\mathrm{sign}(x) \in \{+1, -1\}$ 
as follows.
Let $J_0(x)$ (resp.\ $J_1(x)$)
be a short arc consisting
of the definite fold points (resp.\ 
indefinite fold points) of $\tilde{g}$
near $x$. Then the arcs $\tilde{g}(J_0(x))$
and $\tilde{g}(J_1(x))$ are situated 
in $\R^2$ near $\tilde{g}(x)$ as depicted
in \fullref{fig6-1}, and we define the sign
as in the figure.

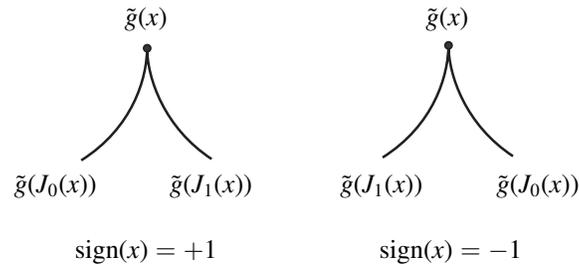
\begin{figure}[ht!]\small
\begin{center}
\labellist\hair 5pt
\pinlabel {$\tilde{g}(x)$} [b] at 155 705
\pinlabel {$\tilde{g}(x)$} [b] at 512 705
\pinlabel {$\tilde{g}(J_0(x))$} [t] <-10pt,0pt> at 079 566
\pinlabel {$\tilde{g}(J_0(x))$} [t] <10pt,0pt> at 588 566
\pinlabel {$\tilde{g}(J_1(x))$} [t] at 232 566
\pinlabel {$\tilde{g}(J_1(x))$} [t] at 434 566
\pinlabel {$\mathrm{sign}(x) = +1$} [t] <0pt,-25pt> at 155 566
\pinlabel {$\mathrm{sign}(x) = -1$} [t] <0pt,-25pt> at 516 566 
\endlabellist
\includegraphics[width=0.5\linewidth]{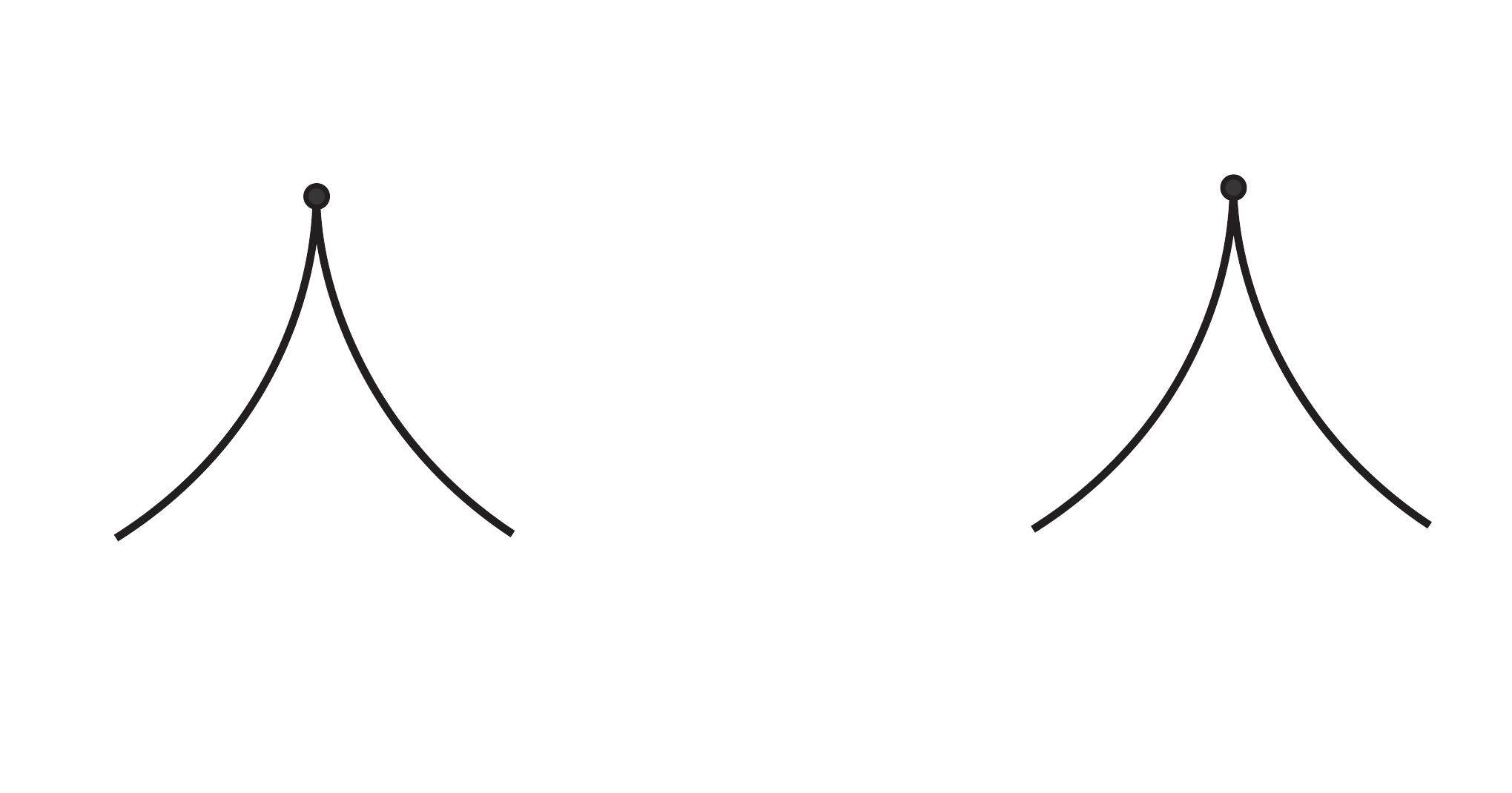}
\end{center}
\caption{Sign for a cusp point $x$}
\label{fig6-1}
\end{figure}

Note that by virtue of condition (G4) above,
$\tilde{g}$ has finitely many cusp points.
The total number of cusp points of
$\tilde{g}$, counted with signs, is called
the \emph{algebraic number of cusps} of $\tilde{g}$.
\eject

\begin{dfn}\label{dfn:topA}
Let $g$ and $g' \co (\R^3, 0) \to (\R^2, 0)$
be smooth map germs. We say that they
are \emph{topologically $\mathcal{A}$--equivalent}
if there exist homeomorphism germs
$\Phi \co (\R^3, 0) \to (\R^3, 0)$ and
$\varphi \co (\R^2, 0) \to (\R^2, 0)$
such that $g' = \varphi^{-1} \circ g \circ \Phi$.
Furthermore, if the homeomorphism germ
$\varphi$ can be chosen so that it preserves the
orientation of $\R^2$, then we say that $g$
and $g'$ are
\emph{topologically $\mathcal{A}_+$--equivalent}.
\end{dfn}

The main result of this section is the
following.

\begin{thm}\label{thm:new}
Let $g \co (\R^3, 0) \to (\R^2, 0)$ be
a generic smooth map germ. Then
the algebraic number of
cusps of a $C^\infty$ stable
perturbation $\tilde{g}$ of a representative of $g$
is an invariant of the topological
$\mathcal{A}_+$--equivalence class of $g$.
In particular, the absolute value of the
algebraic number of cusps of $\tilde{g}$
is an invariant of the topological
$\mathcal{A}$--equivalence class of $g$.
\end{thm}

In order to prove the above theorem, let
us first consider the
following situation.
Let $F \co W \to D^2$ be a smooth
map of a compact $3$--dimensional manifold
$W$ with nonempty boundary $\partial W = M$
with the following properties:
\begin{itemize}
\item[(1)] $F^{-1}(\partial D^2) = M$,
\item[(2)] $f = F|_M \co M \to \partial D^2 = S^1$ is $C^\infty$
stable,
\item[(3)] $F|_{M \times [0, 1)}
= f \times \mathrm{id}_{[0, 1)}$,
where we identify the small open collar neighborhood
of $M$ (or $\partial D^2$) in $W$ (resp.\ in
$D^2$) with $M \times [0, 1)$ (resp.\ 
$\partial D^2 \times [0, 1)$), and
\item[(4)] $F|_{\Int{W}} \co \Int{W}
\to \Int{D^2}$ is a proper $C^\infty$ stable
map.
\end{itemize}
Note that $F$ may have cusp
singular points. Thus, in general,
$F$ has singular fibers
as depicted in \fullref{fig43}.

In \fullref{section4}, we have seen that
the fibers $\tilde{\mathbf{0}}_*$,
$\tilde{\mathrm{I}}^0_*$,
$\tilde{\mathrm{I}}^1_*$ and
$\tilde{\mathrm{II}}^{01}_*$
are co-orientable. If cusp singular
points are allowed, then we see easily that
$\tilde{\mathrm{II}}^a_*$ is also co-orientable.
We give co-orientations to
$\tilde{\mathrm{II}}^a_*$ 
as depicted in \fullref{fig5}.

\begin{figure}[ht!]\small
\begin{center}
\labellist
\pinlabel {$\tilde{\mathrm{II}}^a_{\mathrm{o}}$} [b] <2pt,-2pt> at 147 626
\pinlabel {$\tilde{\mathrm{II}}^a_{\mathrm{e}}$} [b] <2pt,-2pt> at 438 626
\pinlabel {$\tilde{\mathrm{I}}^0_{\mathrm{o}}$} [t] at 361 467
\pinlabel {$\tilde{\mathrm{I}}^1_{\mathrm{o}}$} [t] <2pt,0pt> at 224 467
\pinlabel {$\tilde{\mathrm{I}}^0_{\mathrm{e}}$} [t] at 071 467
\pinlabel {$\tilde{\mathrm{I}}^1_{\mathrm{e}}$} [t] <2pt,0pt> at 514 467
\endlabellist
\includegraphics[width=0.5\linewidth]{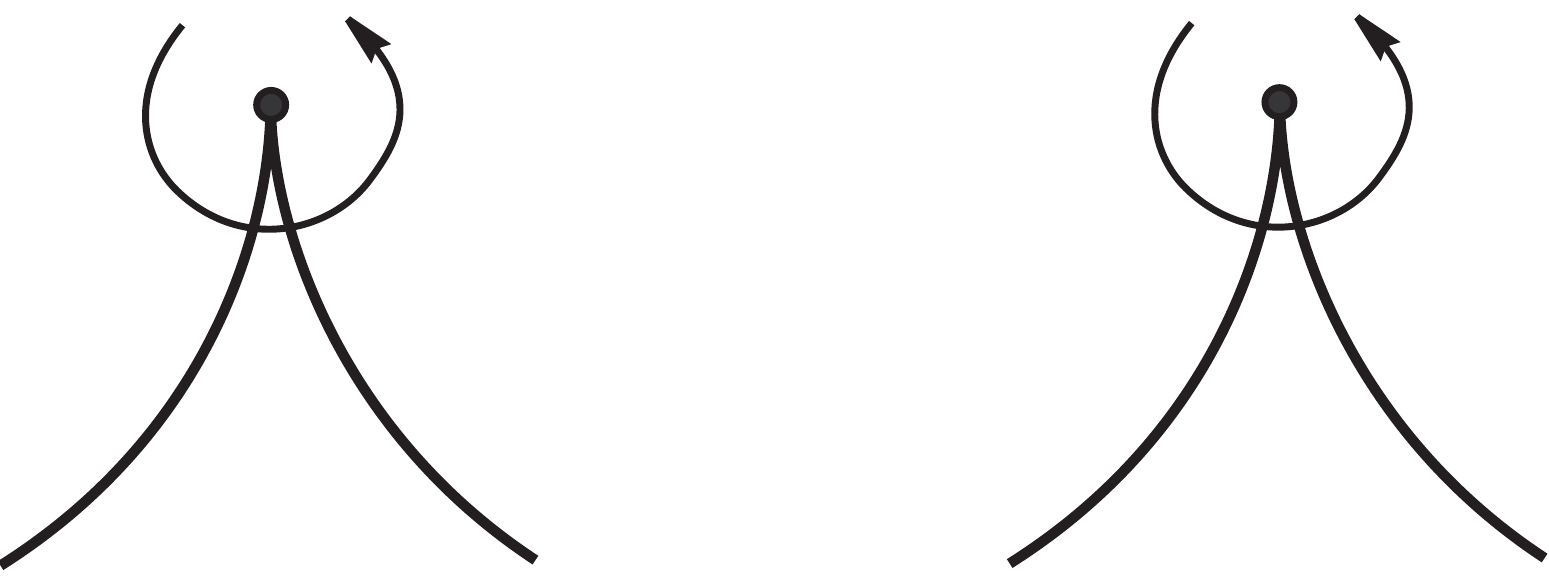}\vspace{2mm}
\end{center}
\caption{Co-orientations for 
$\tilde{\mathrm{II}}^a_{\mathrm{o}}$ and
$\tilde{\mathrm{II}}^a_{\mathrm{e}}$}
\label{fig5}
\end{figure}

In the following, we orient
$D^2$ and $S^1 = \partial D^2$
consistently so that $S^1$ gets the
counterclockwise orientation.
Then we have the following.

\begin{lem}
For the algebraic numbers
of singular fibers of $F$ and $f$,
we have the following:
\begin{equation*}
\begin{split}
||\tilde{\mathrm{I}}^0_{\mathrm{o}}(f)|| &
= -||\tilde{\mathrm{II}}^{01}_{\mathrm{o}}(F)||
+ ||\tilde{\mathrm{II}}^{01}_{\mathrm{e}}(F)||
- ||\tilde{\mathrm{II}}^a_{\mathrm{e}}(F)||, \\
||\tilde{\mathrm{I}}^0_{\mathrm{e}}(f)|| &
= -||\tilde{\mathrm{II}}^{01}_{\mathrm{o}}(F)||
+ ||\tilde{\mathrm{II}}^{01}_{\mathrm{e}}(F)||
+ ||\tilde{\mathrm{II}}^a_{\mathrm{o}}(F)||, \\
||\tilde{\mathrm{I}}^1_{\mathrm{o}}(f)|| &
= ||\tilde{\mathrm{II}}^{01}_{\mathrm{o}}(F)||
- ||\tilde{\mathrm{II}}^{01}_{\mathrm{e}}(F)||
- ||\tilde{\mathrm{II}}^a_{\mathrm{o}}(F)||, \\
||\tilde{\mathrm{I}}^1_{\mathrm{e}}(f)|| &
= ||\tilde{\mathrm{II}}^{01}_{\mathrm{o}}(F)||
- ||\tilde{\mathrm{II}}^{01}_{\mathrm{e}}(F)||
+ ||\tilde{\mathrm{II}}^a_{\mathrm{e}}(F)||.
\end{split}
\end{equation*}
\end{lem}

\begin{proof}
Let us consider the 
closures of $\tilde{\mathrm{I}}^0_{\mathrm{o}}(F)$,
$\tilde{\mathrm{I}}^0_{\mathrm{e}}(F)$,
$\tilde{\mathrm{I}}^1_{\mathrm{o}}(F)$ and
$\tilde{\mathrm{I}}^1_{\mathrm{e}}(F)$ as
$1$--dimen\-sion\-al chains in $D^2$ with coefficients
in $\Z$.
Then by observing the adjacencies
for the singular fibers as we did
to obtain the formulae for
the coboundary homomorphism in
\eqref{eq:delta}, we get the
following equalities
as $0$--dimensional chains:
\begin{equation*}
\begin{split}
\partial \overline{\tilde{\mathrm{I}}^0_{\mathrm{o}}(F)}
& = \tilde{\mathrm{II}}^{01}_{\mathrm{o}}(F)
- \tilde{\mathrm{II}}^{01}_{\mathrm{e}}(F)
+ \tilde{\mathrm{II}}^a_{\mathrm{e}}(F)
+ \tilde{\mathrm{I}}^0_{\mathrm{o}}(f), \\
\partial \overline{\tilde{\mathrm{I}}^0_{\mathrm{e}}(F)}
& = \tilde{\mathrm{II}}^{01}_{\mathrm{o}}(F)
- \tilde{\mathrm{II}}^{01}_{\mathrm{e}}(F)
- \tilde{\mathrm{II}}^a_{\mathrm{o}}(F)
+ \tilde{\mathrm{I}}^0_{\mathrm{e}}(f), \\
\partial \overline{\tilde{\mathrm{I}}^1_{\mathrm{o}}(F)}
& = -\tilde{\mathrm{II}}^{01}_{\mathrm{o}}(F)
+ \tilde{\mathrm{II}}^{01}_{\mathrm{e}}(F)
+ \tilde{\mathrm{II}}^a_{\mathrm{o}}(F)
+ \tilde{\mathrm{I}}^1_{\mathrm{o}}(f), \\
\partial \overline{\tilde{\mathrm{I}}^1_{\mathrm{e}}(F)}
& = -\tilde{\mathrm{II}}^{01}_{\mathrm{o}}(F)
+ \tilde{\mathrm{II}}^{01}_{\mathrm{e}}(F)
- \tilde{\mathrm{II}}^a_{\mathrm{e}}(F)
+ \tilde{\mathrm{I}}^1_{\mathrm{e}}(f).
\end{split}
\end{equation*}
Since the algebraic number of
points in the boundary of a $1$--dimensional
chain is always equal to zero,
we get the desired equalities.
\end{proof}\eject

\begin{rmk}\label{rmk:newproof}
By the above lemma, we see easily that
$$
||\tilde{\mathrm{I}}^0_{\mathrm{o}}(f)||
+ ||\tilde{\mathrm{I}}^1_{\mathrm{e}}(f)|| = 0
\quad \text{ and } \quad
||\tilde{\mathrm{I}}^0_{\mathrm{e}}(f)||
+ ||\tilde{\mathrm{I}}^1_{\mathrm{o}}(f)|| = 0.
$$
This gives an alternative proof 
of the fact that $\beta_1(f) = 0$
for a $C^\infty$ stable map $f$
of a closed surface into $S^1$ or into $\R$
$(\subset S^1)$,
where $\beta_1$ is the
cohomology class described
in \fullref{lem:coho}
(see also \cite[Lemma~14.1]{Saeki04}).
\end{rmk}

For $\tilde{\mathrm{II}}^a$,
we consider the co-orientation
as depicted in \fullref{fig6}.
Since we have
$$||\tilde{\mathrm{II}}^a(F)||
= ||\tilde{\mathrm{II}}^a_{\mathrm{o}}(F)||
+ ||\tilde{\mathrm{II}}^a_{\mathrm{e}}(F)||,$$
we immediately get the following.

\begin{figure}[ht!]\small
\begin{center}
\labellist\hair 5pt
\pinlabel {$\tilde{\mathrm{II}}^a$} [b] <3pt,0pt> at 281 709
\pinlabel {$\tilde{\mathrm{I}}^0$} [t] at 204 558
\pinlabel {$\tilde{\mathrm{I}}^1$} [t] <2pt,0pt> at 358 558
\endlabellist
\includegraphics[width=0.2\linewidth]{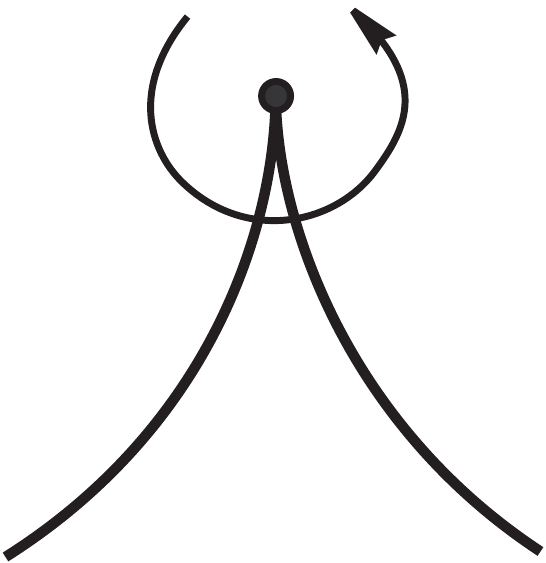}\vspace{2mm}
\end{center}
\caption{Co-orientation for 
$\tilde{\mathrm{II}}^a$}
\label{fig6}
\end{figure}

\begin{prop}\label{prop:cusp}
The algebraic number of singular fibers
of $F$ containing cusps is equal to
\begin{equation}
-||\tilde{\mathrm{I}}^0_{\mathrm{o}}(f)||
+ ||\tilde{\mathrm{I}}^0_{\mathrm{e}}(f)||
= -||\tilde{\mathrm{I}}^1_{\mathrm{o}}(f)||
+ ||\tilde{\mathrm{I}}^1_{\mathrm{e}}(f)||.
\label{eq:cob}
\end{equation}
\end{prop}

Note that the integer given by \eqref{eq:cob}
is a fold cobordism invariant as shown
in \fullref{section4}.

Now, let $g \co (\R^3, 0) \to (\R^2, 0)$
be a generic smooth map germ.
Suppose that the origin
is isolated in $g^{-1}(0)$, i.e.\ 
$0 \not\in \overline{g^{-1}(0) \setminus
\{0\}}$.
Then for $\varepsilon > 0$ sufficiently
small, $\tilde{S}^2_\varepsilon
= g^{-1}(S^1_\varepsilon)$ is
diffeomorphic to $S^2$, and $g|_{g^{-1}
(D^2_\varepsilon \setminus \{0\})}$
is $C^\infty$ equivalent to $g_\partial
\times \id_{(0, \varepsilon)}$, where
$$g_\partial = g|_{g^{-1}(S^1_\varepsilon)} \co
\tilde{S}^2_\varepsilon \to S^1_\varepsilon$$
is a $C^\infty$ stable map.

Then, we have the following.
\eject

\begin{prop}\label{thm:inv}
Let $g \co (\R^3, 0) \to (\R^2, 0)$ be
a generic smooth map germ such that $0$
is isolated in $g^{-1}(0)$. Then
the algebraic number of
cusps of a $C^\infty$ stable
perturbation $\tilde{g}$ of a representative of $g$
is an invariant of the topological
$\mathcal{A}_+$--equivalence class of $g$, and is
equal to
\begin{equation}
-||\tilde{\mathrm{I}}^0_{\mathrm{o}}(g_\partial)||
+ ||\tilde{\mathrm{I}}^0_{\mathrm{e}}(g_\partial)||
= -||\tilde{\mathrm{I}}^1_{\mathrm{o}}(g_\partial)||
+ ||\tilde{\mathrm{I}}^1_{\mathrm{e}}(g_\partial)||,
\label{eq:inv}
\end{equation}
where 
$$g_\partial = g|_{g^{-1}(S^1_\varepsilon)} \co
\tilde{S}^2_\varepsilon \to S^1_\varepsilon$$
and $\varepsilon > 0$ is sufficiently small.
In particular, the absolute value of the
algebraic number of cusps of $\tilde{g}$
is an invariant of the topological
$\mathcal{A}$--equivalence class of $g$.
\end{prop}

For the proof, we need the following observation.

\begin{rmk}\label{rmk:top}
The notion of a singular fiber and the
corresponding $C^0$ equivalence relation
can be generalized to
continuous maps between topological spaces.
In particular, the theory as developed in
\cite{Saeki04} can be generalized to
proper continuous maps between topological
manifolds which are topologically equivalent
to smooth Thom maps between smooth manifolds, 
and their topological cobordisms.
This is because the classification of
singular fibers is based on the ``$C^0$ equivalence''
and not on the ``$C^\infty$ equivalence''
in the theory developed in \cite{Saeki04}.
\end{rmk}

\begin{proof}[Proof of Proposition~\textup{\ref{thm:inv}}]
Let $g' \co (\R^3, 0) \to (\R^2, 0)$ be a smooth
map germ which is generic and is topologically
$\mathcal{A}_+$--equivalent to $g$.
Thus there exist homeomorphism germs
$\Phi \co (\R^3, 0) \to (\R^3, 0)$ and
$\varphi \co (\R^2, 0) \to (\R^2, 0)$
such that $g' = \varphi^{-1} \circ g \circ \Phi$
and $\varphi$ is orientation preserving.

Let $\varepsilon$ (resp.\ 
$\varepsilon'$) be a small
positive real number as above
for $g$ (resp.\ for $g'$). 
We can take $\varepsilon'$
sufficiently small so that
$\varphi(D^2_{\varepsilon'})
\subset \Int{D^2_\varepsilon}$.
Set 
$$Y = D^2_\varepsilon \setminus
\varphi(\Int{D^2_{\varepsilon'}}) \quad
\text{ and } \quad W = g^{-1}(Y).$$ 
We can
show that $Y$ is homeomorphic
to $S^1 \times [0, 1]$ and
$W$ is homeomorphic to $S^2 \times [0, 1]$.
Note that $g^{-1}(\varphi(S^1_{\varepsilon'}))
= \Phi(g'^{-1}(S^1_{\varepsilon'}))$.

Then the map $g|_W \co W \to Y$ gives a
(fold) cobordism between 
$$g_\partial = g|_{g^{-1}(S^1_\varepsilon)} \co
g^{-1}(S^1_\varepsilon) \to S^1_\varepsilon$$
and 
\begin{equation}
\varphi \circ g'_\partial \circ
\Phi^{-1} \co \Phi(g'^{-1}(S^1_{\varepsilon'}))
\to \varphi(S^1_{\varepsilon'}).
\label{eq:g'}
\end{equation}
(Precisely speaking, $W$ and $Y$ are only topological
manifolds with boundary and $g|_W$
is merely a continuous map. Nevertheless,
we can regard $g|_W$  as a topological fold cobordism
in an appropriate sense. See \fullref{rmk:top}.)

Then by the results obtained in \fullref{section4},
we have\footnote{In \fullref{section4} we have considered
Morse functions on surfaces: however, almost all
the arguments work also for $C^\infty$ stable maps
into $S^1$.}
$$-||\tilde{\mathrm{I}}^0_{\mathrm{o}}(g_\partial)||
+ ||\tilde{\mathrm{I}}^0_{\mathrm{e}}(g_\partial)||
= -||\tilde{\mathrm{I}}^0_{\mathrm{o}}(g'_\partial)||
+ ||\tilde{\mathrm{I}}^0_{\mathrm{e}}(g'_\partial)||.$$
This is because $\varphi$ preserves the orientation of
$\R^2$ and hence
the algebraic number of singular fibers of
a given type for the map \eqref{eq:g'} is
equal to that for $g'_\partial$.
Therefore, the integer \eqref{eq:inv}
is an invariant of the topological
$\mathcal{A}_+$--equivalence class.
(Note that the equality in \eqref{eq:inv}
follows from \fullref{prop:cusp}.)

On the other hand, let $\tilde{g} \co U \to \Int{D^2_\varepsilon}$
be a proper $C^\infty$ stable perturbation of $g|_U$,
where $U = g^{-1}(\Int{D^2_\varepsilon})$.
Set 
$$U' = g^{-1}(\Int{D^2_\varepsilon} \setminus
D^2_{\varepsilon/2}) \quad
\text{ and } \quad U'' = \tilde{g}^{-1}
(\Int{D^2_\varepsilon} \setminus
D^2_{\varepsilon/2}).$$
Then, by virtue of conditions (G2) and (G4) of
\fullref{dfn:generic}, we see that
$g|_{U'}$ and $\tilde{g}|_{U''}$ are
topologically $\mathcal{A}_+$--equivalent
in a sense similar to \fullref{dfn:topA}.
Then we see that
the integer \eqref{eq:inv}
is equal to the algebraic number of
cusps of $\tilde{g}$
by \fullref{prop:cusp}.

It is easy to observe that if we
reverse the orientation of $\R^2$, then
the algebraic number of cusps changes the sign.
Thus the last assertion of \fullref{thm:inv}
follows immediately.
This completes the proof.
\end{proof}

In order to generalize the above result
to the case where the origin may
not necessarily be isolated in
$g^{-1}(0)$, let us consider the
following situation.
Let $F \co W \to D^2$ be a smooth
map of a compact $3$--dimensional manifold
$W$ with nonempty boundary $\partial W$
with the following properties:
\begin{itemize}
\item[(1)] $\partial W = M \cup P$, where
$M$ is a compact surface with boundary,
$P$ is a finite disjoint union of $2$--dimensional
disks, and $M \cap P = \partial M = \partial P$,
\item[(2)] $F^{-1}(\partial D^2) = M$,
\item[(3)] $F|_P \co P \to D^2$ is a submersion
so that, in particular, it is a diffeomorphism
on each component of $P$,
\item[(4)] $f = F|_M \co M \to \partial D^2 = S^1$ is $C^\infty$
stable,
\item[(5)] $F|_{M \times [0, 1)}
= f \times \mathrm{id}_{[0, 1)}$,
where we identify the small open ``collar neighborhood''
of $M$ (or $\partial D^2$) in $W$ (resp.\ in
$D^2$) with $M \times [0, 1)$ (resp.\ 
$\partial D^2 \times [0, 1)$), and
\item[(6)] $F|_{W \setminus M} \co W \setminus M
\to \Int{D^2}$ is a proper $C^\infty$ stable
map.
\end{itemize}

In what follows, for simplicity
we assume that $W$ and $M$ are
orientable, which is
enough for our purpose.

Then we can get a list of the $C^0$ equivalence
classes of singular fibers that appear
for $F$ as above, which is similar
to \fullref{fig43}. For these fibers,
let us consider the following
equivalence relation: two fibers are equivalent
if one is
$C^0$ equivalent to the other one after adding
even numbers of regular components to
both of the fibers. Note that
in contrast to the case where
$\partial M = \emptyset$, regular fibers
consist of circles and intervals. However,
when we count the number of regular
components, we do not distinguish them.

Then we easily get the following.

\begin{lem}
Those equivalence classes of singular fibers
which are co-orientable are $\tilde{\mathfrak{F}}_\ast$,
where $\ast = {\mathrm{o}}$ or ${\mathrm{e}}$,
and $\tilde{\mathfrak{F}}$ are
as depicted in \fullref{fig7}.
\end{lem}

\begin{figure}[ht!]\small
\centering
\labellist
\pinlabel {$\kappa = 1$} [rB] <-30pt,0pt> at -75 733
\pinlabel {$\kappa = 2$} [rB] <-30pt,0pt> at -75 604
\pinlabel {$\tilde{\mathrm{I}}^0$} [rB] at -12 733
\pinlabel {$\tilde{\mathrm{I}}^1$} [rB] at 162 733
\pinlabel {$\tilde{\mathrm{I}}^\alpha$} [rB] at 410 733
\pinlabel {$\tilde{\mathrm{II}}^{01}$} [rB] at -75 604
\pinlabel {$\tilde{\mathrm{II}}^{0\alpha}$} [rB] at 162 604
\pinlabel {$\tilde{\mathrm{II}}^{1\alpha}$} [rB] at 410 604
\pinlabel {$\tilde{\mathrm{II}}^\beta$} [r] at -40 470
\pinlabel {$\tilde{\mathrm{II}}^\gamma$} [r] at 410 470
\pinlabel {$\tilde{\mathrm{II}}^a$} [r] at 162 470
\endlabellist
\includegraphics[width=0.65\linewidth]{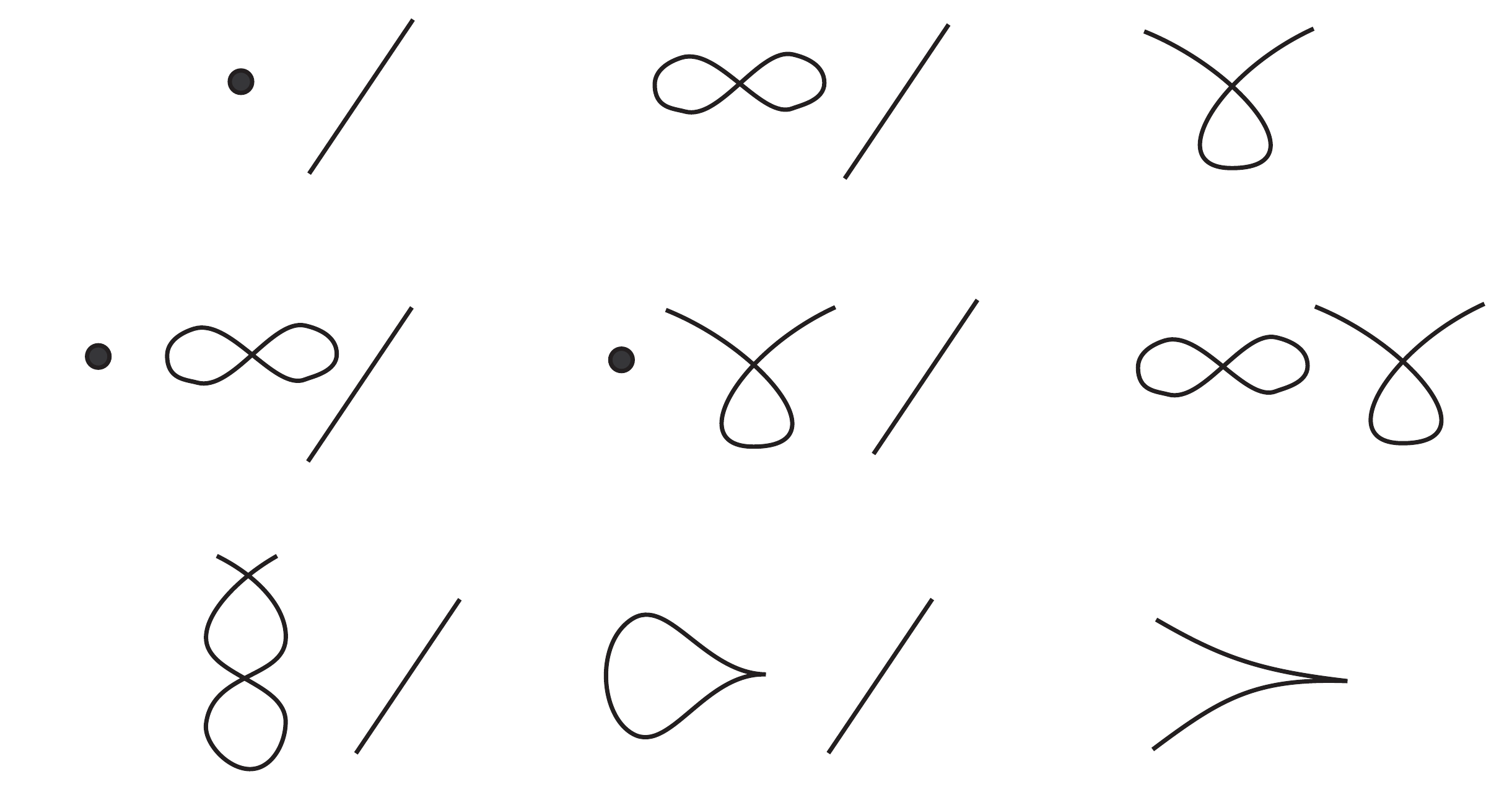}
\caption{List of co-orientable singular fibers for $F$}
\label{fig7}
\end{figure}
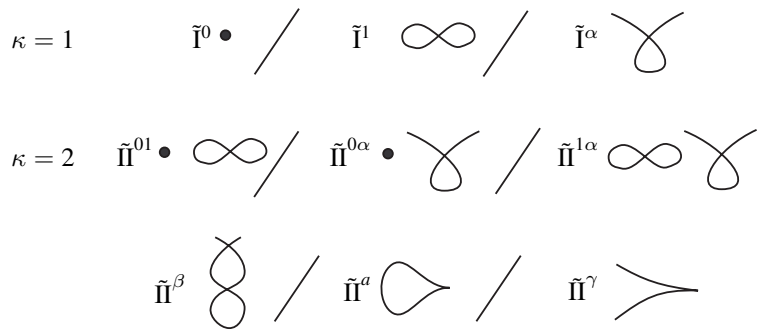

We denote by $\tilde{\mathbf{0}}_\ast$
the equivalence classes corresponding to
regular fibers. Note that they
are also co-orientable.
We fix a co-orientation for each
co-orientable equivalence
class of singular fibers as above.
Then for the coboundary
homomorphism, we get the following:
\begin{equation*}
\begin{split}
\delta_1(\tilde{\mathrm{I}}^0_{\mathrm{o}}) &
= \tilde{\mathrm{II}}^{01}_{\mathrm{o}}
- \tilde{\mathrm{II}}^{01}_{\mathrm{e}}
- \tilde{\mathrm{II}}^a_{\mathrm{e}}
- \tilde{\mathrm{II}}^{0\alpha}_{\mathrm{o}}
+ \tilde{\mathrm{II}}^{0\alpha}_{\mathrm{e}}
- \tilde{\mathrm{II}}^\gamma_{\mathrm{e}}, \\
\delta_1(\tilde{\mathrm{I}}^0_{\mathrm{e}}) &
= \tilde{\mathrm{II}}^{01}_{\mathrm{o}}
- \tilde{\mathrm{II}}^{01}_{\mathrm{e}}
+ \tilde{\mathrm{II}}^a_{\mathrm{o}}
- \tilde{\mathrm{II}}^{0\alpha}_{\mathrm{o}}
+ \tilde{\mathrm{II}}^{0\alpha}_{\mathrm{e}}
+ \tilde{\mathrm{II}}^\gamma_{\mathrm{o}}, \\
\delta_1(\tilde{\mathrm{I}}^1_{\mathrm{o}}) &
= - \tilde{\mathrm{II}}^{01}_{\mathrm{o}}
+ \tilde{\mathrm{II}}^{01}_{\mathrm{e}}
- \tilde{\mathrm{II}}^a_{\mathrm{o}}
- \tilde{\mathrm{II}}^{1\alpha}_{\mathrm{o}}
+ \tilde{\mathrm{II}}^{1\alpha}_{\mathrm{e}}
- \tilde{\mathrm{II}}^\beta_{\mathrm{e}}, 
\end{split}
\end{equation*}

\begin{equation*}
\begin{split}
\delta_1(\tilde{\mathrm{I}}^1_{\mathrm{e}}) &
= - \tilde{\mathrm{II}}^{01}_{\mathrm{o}}
+ \tilde{\mathrm{II}}^{01}_{\mathrm{e}}
+ \tilde{\mathrm{II}}^a_{\mathrm{e}}
- \tilde{\mathrm{II}}^{1\alpha}_{\mathrm{o}}
+ \tilde{\mathrm{II}}^{1\alpha}_{\mathrm{e}}
+ \tilde{\mathrm{II}}^\beta_{\mathrm{o}}, \\
\delta_1(\tilde{\mathrm{I}}^\alpha_{\mathrm{o}}) &
= \tilde{\mathrm{II}}^{0\alpha}_{\mathrm{o}}
- \tilde{\mathrm{II}}^{0\alpha}_{\mathrm{e}}
+ \tilde{\mathrm{II}}^{1\alpha}_{\mathrm{o}}
- \tilde{\mathrm{II}}^{1\alpha}_{\mathrm{e}}
+ \tilde{\mathrm{II}}^\beta_{\mathrm{e}}
- \tilde{\mathrm{II}}^\gamma_{\mathrm{o}}, \\
\delta_1(\tilde{\mathrm{I}}^\alpha_{\mathrm{e}}) &
= \tilde{\mathrm{II}}^{0\alpha}_{\mathrm{o}}
- \tilde{\mathrm{II}}^{0\alpha}_{\mathrm{e}}
+ \tilde{\mathrm{II}}^{1\alpha}_{\mathrm{o}}
- \tilde{\mathrm{II}}^{1\alpha}_{\mathrm{e}}
- \tilde{\mathrm{II}}^\beta_{\mathrm{o}}
+ \tilde{\mathrm{II}}^\gamma_{\mathrm{e}}.
\end{split}
\end{equation*}
Then by the same argument as before,
we see that 
\begin{equation}
||\tilde{\mathrm{I}}^\alpha_{\mathrm{o}}(f)||
- ||\tilde{\mathrm{I}}^\alpha_{\mathrm{e}}(f)||
+ ||\tilde{\mathrm{I}}^1_{\mathrm{o}}(f)||
- ||\tilde{\mathrm{I}}^1_{\mathrm{e}}(f)||
\quad \text{ and } \quad
- ||\tilde{\mathrm{I}}^0_{\mathrm{o}}(f)||
+ ||\tilde{\mathrm{I}}^0_{\mathrm{e}}(f)||
\label{eq:cusps}
\end{equation}
are fold cobordism invariants of $f$
in the following sense. 

\begin{dfn}\label{dfn:cob-b}
Let $f_i \co M_i \to S^1$, $i = 0, 1$, be
proper $C^\infty$ stable maps of compact
surfaces with boundary such that $f_i|_{\partial M_i}$
are submersions. We say that $f_0$ and $f_1$ are
(fold) \emph{cobordant} if there exist a compact
$3$--dimensional manifold $X$ with corners and a fold
map $F \co X \to S^1 \times [0, 1]$ with
the following properties:
\begin{itemize}
\item[(1)] $\partial X = M_0 \cup Q \cup M_1$,
where $M_0 \cap M_1 = \emptyset$ and $Q$
is a compact surface with boundary $\partial
Q = (Q \cap M_0) \cup (Q \cap M_1)$,
\item[(2)] $X$ has corners along $\partial Q$,
\item[(3)] $F|_Q \co Q \to S^1 \times [0, 1]$ is
a submersion, and
\item[(4)] we have
\begin{align*}
F|_{M_0 \times [0, \varepsilon)} & = f_0 \times \id_{[0,
\varepsilon)} \co M_0 \times [0, \varepsilon) \to S^1 \times
[0, \varepsilon), \quad \text{\rm and} \\
F|_{M_1 \times (1-\varepsilon, 1]} 
& = f_1 \times \id_{(1-\varepsilon,
1]} \co M_1 \times (1-\varepsilon, 1] \to S^1 \times
(1-\varepsilon, 1]
\end{align*}
for some sufficiently small $\varepsilon > 0$, where
we identify the open 
``collar neighborhoods'' of $M_0$ and $M_1$ in
$X$ with $M_0 \times [0, \varepsilon)$ and $M_1 
\times (1-\varepsilon,
1]$ respectively.
\end{itemize}
\end{dfn}

Furthermore,
we see that the algebraic number of
cusps of $F$ given by
$$
||\tilde{\mathrm{II}}^a_{\mathrm{o}}(F)||
+ ||\tilde{\mathrm{II}}^a_{\mathrm{e}}(F)||
+ ||\tilde{\mathrm{II}}^\gamma_{\mathrm{o}}(F)||
+ ||\tilde{\mathrm{II}}^\gamma_{\mathrm{e}}(F)||
$$
is equal to both of the integers \eqref{eq:cusps}.

Now, let $g \co (\R^3, 0) \to (\R^2, 0)$ be
a generic smooth map germ.
Then by applying the above
observations to the stable perturbation of the map
$$g|_{D^3_\delta \cap g^{-1}(\Int{D^2_\varepsilon})} \co
D^3_\delta \cap g^{-1}(\Int{D^2_\varepsilon}) \to 
\Int{D^2_\varepsilon},$$
we get the following, where $\delta$ and $\varepsilon$
are as in \fullref{dfn:generic}.\eject

\begin{prop}\label{thm:inv2}
Let $g \co (\R^3, 0) \to (\R^2, 0)$ be
a generic smooth map germ. Then
the algebraic number of
cusps of a $C^\infty$ stable
perturbation $\tilde{g}$ of a representative of $g$
is an invariant of the topological
$\mathcal{A}_+$--equivalence class of $g$, and is
equal to
\begin{equation*}
-||\tilde{\mathrm{I}}^0_{\mathrm{o}}(g_\partial)||
+ ||\tilde{\mathrm{I}}^0_{\mathrm{e}}(g_\partial)||
= ||\tilde{\mathrm{I}}^\alpha_{\mathrm{o}}(g_\partial)||
- ||\tilde{\mathrm{I}}^\alpha_{\mathrm{e}}(g_\partial)||
+ ||\tilde{\mathrm{I}}^1_{\mathrm{o}}(g_\partial)||
- ||\tilde{\mathrm{I}}^1_{\mathrm{e}}(g_\partial)||,
\end{equation*}
where 
$$g_\partial = g|_{D^3_\delta \cap g^{-1}(S^1_\varepsilon)} \co
D^3_\delta \cap g^{-1}(S^1_\varepsilon) \to S^1_\varepsilon$$
and $0 < \varepsilon << \delta$ are sufficiently small.
In particular, the absolute value of the
algebraic number of cusps of $\tilde{g}$
is an invariant of the topological
$\mathcal{A}$--equivalence class of $g$.
\end{prop}

\begin{proof}
Let $g' \co (\R^3, 0) \to (\R^2, 0)$ be a smooth
map germ which is generic and is topologically
$\mathcal{A}_+$--equivalent to $g$.
Thus there exist homeomorphism germs
$\Phi \co (\R^3, 0) \to (\R^3, 0)$ and
$\varphi \co (\R^2, 0) \to (\R^2, 0)$
such that $g' = \varphi^{-1} \circ g \circ \Phi$
and $\varphi$ is orientation preserving.

Let $\delta$ and $\varepsilon$ (resp.\ 
$\delta'$ and $\varepsilon'$) be small
positive real numbers as in \fullref{dfn:generic}
for $g$ (resp.\ for $g'$). 
We can take $\delta'$ and $\varepsilon'$
sufficiently small so that
$\varphi(D^2_{\varepsilon'})
\subset \Int{D^2_\varepsilon}$
and $\Phi(g'^{-1}(D^2_{\varepsilon'}) \cap D^3_{\delta'})
\subset g^{-1}(\Int{D^2_\varepsilon}) \cap 
\Int{D^3_\delta}$.

Set 
$$\bar{g}_{\partial}
= g|_{g^{-1}(\varphi(S^1_{\varepsilon'})) \cap D^3_\delta}
\co g^{-1}(\varphi(S^1_{\varepsilon'})) \cap D^3_\delta
\to \varphi(S^1_{\varepsilon'}).$$
It is easy to observe that for each 
$y \in \varphi(S^1_{\varepsilon'})$,
the fiber
of $\bar{g}_\partial$ over $y$ and that of
$$\varphi \circ g'_\partial \circ \Phi^{-1} \co
\Phi(g'^{-1}(S^1_{\varepsilon'}) \cap D^3_{\delta'})
\to \varphi(S^1_{\varepsilon'})$$
over $y$ is $C^0$ equivalent.

Set $Y = D^2_\varepsilon \setminus
\varphi(\Int{D^2_{\varepsilon'}})$
and $W = g^{-1}(Y) \cap D^3_\delta$. We see
that $Y$ is homeomorphic
to $S^1 \times [0, 1]$
and $g|_W \co W \to Y$ gives a (fold)
cobordism between $g_\partial$
and $\bar{g}_\partial$ in the sense of
\fullref{dfn:cob-b} (see also
\fullref{rmk:top}).
(Note that the boundary causes no
problem by virtue of condition (G3)
of \fullref{dfn:generic}.)

Then by the same argument as in the proof
of \fullref{thm:inv} together with
the above results, we get the desired result.
\end{proof}

Now \fullref{thm:new} follows from
Propositions~\ref{thm:inv} and \ref{thm:inv2}.

Compare \fullref{thm:new} with the
results obtained in \cite{FI,FBS,BS,Ohsumi},
etc. It would be an interesting problem
to find a formula expressing the algebraic
number of cusps of a $C^\infty$ stable
perturbation in algebraic terms: e.g.\ as
the signature of a certain quadratic form
associated with a generic map germ.

It would also be interesting to find
topological invariants of generic
smooth map germs $(\R^n, 0) \to
(\R^p, 0)$ with $n > p$ arising from
the number of certain singular fibers
of a stable (or Thom--Boardman generic)
perturbation.

\bibliographystyle{gtart}
\bibliography{link}

\end{document}